\documentclass[11pt]{article}
\usepackage{style}
\usepackage[left=1in,bottom=1in,top=1in,right=1in]{geometry}
\renewcommand{\P}{\set{P}}
\newcommand{\+}{\frac{1}{2}+\eps}

\hypersetup{%
	colorlinks   = true, %
	urlcolor     = blue, %
	linkcolor    = blue, %
	citecolor    = blue,  %
}
\title{A Simple Geometric Proof of the Optimality of the Sequential Probability Ratio Test for Symmetric Bernoulli Hypotheses}

\author{%
  \begin{tabular}[t]{@{}c@{\hspace{5em}}c@{\hspace{5em}}c@{}} %
    Chirag Pabbaraju\thanks{\texttt{cpabbara@cs.stanford.edu}} &
    Gregory Valiant\thanks{\texttt{gregory.valiant@gmail.com}} &
    Rishi Verma\thanks{\texttt{verma.rishiraj@gmail.com}}
  \end{tabular}\\[1ex]
  \textit{Computer Science Department, Stanford University}
}

\date{\today}

\begin{document}

\maketitle

\begin{abstract}
    This paper revisits the classical problem of determining the bias of a weighted coin, where the bias is known to be either $p = 1/2 + \varepsilon$ or $p = 1/2 - \varepsilon$, while minimizing the expected number of coin tosses and the error probability. The optimal strategy for this problem is given by Wald's Sequential Probability Ratio Test (SPRT), which compares the log-likelihood ratio against fixed thresholds to determine a stopping time. Classical proofs of this result typically rely on analytical, continuous, and non-constructive arguments. In this paper, we present a discrete, self-contained proof of the optimality of the SPRT for this problem. We model the problem as a biased random walk on the two-dimensional (heads, tails) integer lattice, and model strategies as marked stopping times on this lattice. Our proof takes a straightforward greedy approach, showing how any arbitrary strategy may be transformed into the optimal, parallel-line ``difference policy'' corresponding to the SPRT, via a sequence of local perturbations that improve a Bayes risk objective.
\end{abstract}

\newpage %

\section{Introduction}
\label{sec:intro}

Suppose we are given a biased coin which lands heads with unknown probability $p$, where $p$ is either $1/2 + \eps$ or $1/2 - \eps$ for some known $\eps > 0$. Our goal is to determine the bias of the coin using the minimum number of coin flips in expectation, and with small failure probability. After each coin flip, we can either output our prediction of the bias $p$ of the coin, or defer the decision and collect the result of the next coin flip. 

This is a special case of the classical problem of \textit{sequential hypothesis testing}, whose roots were laid in the pioneering work of Wald \cite{wald1945sequential}. In the problem of sequential testing with a pair of simple hypotheses, we must decide between two hypotheses $H_0$ and $H_1$ over a sequence of observations collected one at a time. The stopping criterion is determined by a pre-defined \textit{policy} for every sequence of observations. %

In this setting, Wald's Sequential Probability Ratio Test (SPRT) \cite{wald1945sequential} has been shown to be optimal \cite{wald1948optimum}, minimizing the expected number of observations collected for any given failure probability. The SPRT is defined by specifying two parameters $A$ and $B$, where $A < 0 < B$. At any time $t$, after having seen observations $x_1,\dots,x_t$, the SPRT computes the log-likelihood ratio of these observations under $H_1$ to that under $H_0$; if this log-likelihood ratio is smaller than $A$, the test stops and accepts $H_0$, and if it is larger than $B$, it accepts $H_1$; otherwise, the test goes on to collect observation $x_{t+1}$.

The aforementioned problem of deciding whether the bias $p$ of a coin is either $1/2+\eps$ or $1/2-\eps$ corresponds to the sequential testing problem with \textit{symmetric Bernoulli hypotheses} $H_0: p=1/2-\eps$ and $H_1: p=1/2+\eps$. In this case, the SPRT, with parameters $A,B$ satisfying $-A=B$, simplifies into an elegant parallel-line \emph{difference} or \emph{linear} policy $\P_c$. In particular, we continue flipping the coin until the absolute difference between the number of heads and tails reaches a particular integer threshold $c=c(A,\eps)$. At this point, we stop and declare the bias to be in the direction of the more frequent outcome: $1/2+\eps$ if there are more heads, and $1/2-\eps$ otherwise.

The optimality of the SPRT is a fundamental result in sequential testing, as it applies generically to every simple hypothesis test. Its proof has been approached from a number of different angles. However, many of these rely on non-constructive continuity-based arguments and complex analysis. %
In the symmetric Bernoulli case to which we restrict our attention, this complex machinery seemingly obscures the intuition for the optimality of the simple difference policy.

In this paper, we present a new, self-contained proof of the optimality of the difference policy for the specific setting of symmetric Bernoulli hypotheses. Our alternate approach is discrete and constructive. It does not require using higher-power tools, and beyond a basic result about stopping times for martingales, relies entirely on elementary arguments and geometric intuition. Specifically, we model the sequence of coin flips as a biased random walk on the two-dimensional integer lattice $\mathbb{Z}_{\geq 0}^2$, corresponding to the (heads, tails) plane. Because the count of heads observed is a sufficient statistic for the bias $p$ of the coin, any policy can be seen as a coloring of this lattice with three colors, corresponding to the decisions ``predict $H_0$'', ``predict $H_1$'', and ``continue tossing'' (see \Cref{fig:policy}). The points where we predict either $H_0$ or $H_1$ correspond to \textit{stopping points} of the policy. Thus, we can frame the sequential testing problem in terms of defining optimal decision boundaries or stopping points on this lattice which minimize the \textit{hitting time} of the random walk. %

Our proof then defines an objective function for any policy, which resembles the \textit{Bayes risk} objective \cite{wald1945sequential} typically introduced in other proofs of optimality for the SPRT. This objective function is a linear combination of the failure probability and the expected number of tosses of the policy. The high-level proof strategy is as follows: suppose we want to establish optimality for a difference policy with threshold $c$. Our approach shows how we can carefully modify any arbitrary policy to this difference policy, by making \textit{local, greedy adjustments} to the lattice representation of the policy. For linear combination parameters chosen as a careful function of $c$, we ensure that each adjustment may only improve the objective function. At the end of the procedure, we will have transformed the policy into the target difference policy, establishing its optimality. We present a more detailed overview of this novel greedy strategy in \Cref{sec:overview}.

\subsection{Main Results}
\label{sec:main-results}

With any policy $\mcP$, we associate its \textit{profile} $(\delta, H)$, which comprises of four quantities:
\begin{alignat}{2}
    \label{eqn:profile}
    &\delta^+ = \Pr_{1/2+\eps}\left[\mcP \text{ predicts } 1/2-\eps\right], \qquad
    &&\delta^- = \Pr_{1/2-\eps}\left[\mcP \text{ predicts } 1/2+\eps\right] \nonumber\\
    &H^+ = \E_{1/2+\eps}\left[\# \text{ tosses by } \mcP\right], 
    &&H^- = \E_{1/2-\eps}\left[\# \text{ tosses by } \mcP\right].
\end{alignat}
Here, the notation $\Pr_{1/2+\eps}[\cdot]$ and $\E_{1/2+\eps}[\cdot]$ denotes the probability and expectation with respect to the randomness of the observed coin flips under $H_1:p=1/2+\eps$ and the policy $\mcP$ (and similarly for the notation $\Pr_{1/2-\eps}[\cdot], \E_{1/2-\eps}[\cdot]$). Recall that a linear policy $\mcP_c$ with threshold $c$ stops when the absolute difference between the observed number of heads and tails equals $c$, and predicts $1/2+\eps$ if the number of heads is larger at this point, and $1/2-\eps$ otherwise.

Then, the first main result we show is the following:

\begin{restatable}[Optimality of the Linear Policy]{theorem}{theoremlinearpolicyoptimality}
\label{thm:linear-policy}
   Consider the sequential hypothesis testing problem of deciding between $H_0:p=1/2-\eps$ and $H_1:p=1/2+\eps$. For any $c \ge 0$, consider the linear policy $\mcP_c$ having profile $(\delta_c, H_c)$. 
   Let $\mcP$ be any other policy with profile $(\delta, H)$.
   Suppose that $\delta^+ + \delta^- < \delta^+_c + \delta^-_c$. Then, it must be the case that $H^+ + H^- > H^+_c + H^-_c$. %
\end{restatable}
This result is indeed a special case implied by the proof of \cite{wald1948optimum}, which establishes optimality of the SPRT. In fact, \Cref{thm:linear-policy} follows as a direct corollary of the following more general result that we show:

\begin{restatable}[Optimality of the Linear Policy for Arbitrary Tradeoffs]{theorem}{theoremlinearpolicyoptimalitybeta}
\label{thm:linear-policy-beta}
   Consider the sequential hypothesis testing problem of deciding between $H_0:p=1/2-\eps$ and $H_1:p=1/2+\eps$. %
   For any tradeoff parameter $\beta > 0$, there exists a corresponding linear policy $\set{P}_c$ with threshold $c=c(\beta)$, and having profile $(\delta_c, H_c)$,
   such that for any other policy $\set{P}$ with profile $(\delta, H)$, 
   it holds that
    \begin{align}
        \label{eqn:bayes-risk-optimality-criterion}
        (\delta^+_c + \delta^-_c) + \beta(H^+_c + H^-_c) \le (\delta^+ + \delta^-) + \beta(H^+ + H^-).
    \end{align}
\end{restatable}
Moreover, the precise linear policy $\mcP_c$ that satisfies the optimality criterion given in \eqref{eqn:bayes-risk-optimality-criterion} for a given value of $\beta$ is easily inferred from our analysis (see \Cref{remark:optimal-policy-characterization}), and can be efficiently determined. Concretely, there is a partitioning of the range $\beta > 0$ into contiguous non-empty intervals $[l_0,u_0], [l_1,u_1],\dots$, such that for any $\beta \in [l_c,u_c]$, $\set{P}_c$ is the linear policy which satisfies \eqref{eqn:bayes-risk-optimality-criterion}. We can then see how the existence of a $\beta$ value in each of these intervals implies \Cref{thm:linear-policy}.

The function $(\delta^+ + \delta^-) + \beta(H^+ + H^-)$ in \eqref{eqn:bayes-risk-optimality-criterion} closely resembles the \textit{Bayes risk} considered in the literature \cite[Chapter 3.12]{lehmann2005testing}. In full generality, the Bayes risk of a policy $\mcP$ is defined with respect to weight parameters $w_0, w_1$ for the failure probabilities, a cost parameter $C$ which accounts for the cost of collecting each observation, and a prior probability $\pi$ of hypothesis $H_0$. The Bayes risk $R(w_0, w_1, C, \pi)$ is then defined as
\begin{align}
    \label{eqn:bayes-risk}
    R(w_0, w_1, C, \pi) := \pi(w_0\delta^- + CH^{-}) + (1-\pi)(w_1\delta^+ + CH^+).
\end{align}
In the special case where $\pi=1/2$ and $w_0=w_1$, the Bayes risk is only a function of $w_0$ and $C$, and minimizing the Bayes risk amounts to minimizing the function $(\delta^+ + \delta^-) + \beta(H^+ + H^-)$, with $\beta=C/w_0$. In this case, \Cref{thm:linear-policy-beta} implies that for every $w_0$ and $C$, there exists a corresponding linear policy $\mcP_c$ (where $c=c(w_0, C)$) that minimizes the Bayes risk $R(w_0, C)$. In other words, \Cref{thm:linear-policy-beta} gives a complete characterization of the specific linear policy that minimizes the Bayes risk corresponding to any parameters $w_0, C$, in the special case where $w_0=w_1$ and the prior $\pi$ is uniform.

\section{Related Work}
\label{sec:related}
The Sequential Probability Ratio Test (SPRT), originally due to \cite{wald1945sequential}, tests between simple hypotheses $H_0: \theta=\theta_0$ and $H_1: \theta=\theta_1$ given a sequence of i.i.d. observations $X_1, X_2, \dots$. At each time step $n$, the test uses the log-likelihood ratio of the sequence of observations under each hypothesis as the test statistic: 
\begin{align}
    \label{eqn:sprt-test-statistic}
    \Lambda_n %
    := \log \frac{\prod_{i=1}^n \Pr_{\theta=\theta_1}[X_i]}{\prod_{i=1}^n \Pr_{\theta=\theta_0}[X_i]}.
\end{align}
The SPRT, defined with parameters $A < 0 < B$,  stops and accepts $H_1$ if $\Lambda_n \geq B$, stops and accepts $H_0$ if $\Lambda_n \leq A$, and continues collecting $X_{n+1}$ otherwise. Wald and Wolfowitz \cite{wald1948optimum} showed that the SPRT is optimal amongst all sequential hypothesis tests, minimizing the expected number of observations for given error probabilities for each hypothesis. In particular, for any SPRT with parameters $A,B$, they show the existence of parameters $w_0,w_1, C$ such that for every prior $\pi$, this SPRT minimizes the Bayes risk $R(w_0, w_1, C, \pi)$ given in \eqref{eqn:bayes-risk}. It has nevertheless been noted in the literature \cite[Chapter 4.1]{shiryaev2011optimal} that given an arbitrary Bayes risk at a specified set of parameters, it is non-trivial to \textit{reverse engineer} the thresholds $A,B$ of the SPRT that is optimal for that Bayes risk, for general simple hypothesis testing tasks. Instead, approximations for $A$ and $B$ are often used. 

For the specific setting of symmetric Bernoulli hypotheses that we focus on, we have $H_0: p=1/2-\eps$ and $H_1: p=1/2+\eps$. Let $\alpha = \frac{1/2+\eps}{1/2-\eps}$. Then, the log-likelihood ratio $\Lambda_n$ after $n$ flips, with $h$ heads, and $t$ tails, is $\log\frac{(1/2+\eps)^{h}(1/2-\eps)^{t}}{(1/2-\eps)^{h}(1/2+\eps)^{t}}=(h-t)\log \alpha$. The SPRT stops when the log-likelihood ratio exceeds or recedes below fixed constants, i.e., when $h - t \ge c_1$ or $h - t \le c_0$ for constants $c_1, c_0$. In particular, if the parameters of the SPRT are equal in magnitude, the test halts when $|h-t|=c$, giving us a symmetric parallel-line policy.

For a history of the evolution and applications of sequential testing, we refer the reader to \cite{ghosh1991brief}. Here, we present a brief summary of past proof techniques of the optimality of the SPRT. Most proof strategies introduce the Bayes risk as an objective function for optimization. Wald and Wolfowitz's original proof minimizes the average risk by defining a gain function to determine whether the information from additional samples is advantageous. They use continuity arguments to show that the gain function is positive only in the closed interval $[A, B]$ for which we continue sampling. A simplified proof by LeCam and Lehmann constructs an auxiliary problem for every SPRT for which it is the optimal strategy \cite{lehmann2005testing}. The proof was further refined by Matthes, who applies a mapping theorem from complex analysis to prove the existence of optimal parameters without explicitly constructing them \cite{matthes1963optimality}.

For the specific case of Bernoulli trials, Girshick identified exact parameters guaranteeing compliance with the failure probability without relying on approximations \cite{girshick1946contributions}. He derived the optimal difference policy, identifying an optimal SPRT for each given choice of failure probability and $\eps$. Girshick reduces the computation of the power curve to solving a system of linear equations. His approach uses generating functions to derive exact expressions for the probability that the sequential test terminates after $n$ steps. In contrast to all these techniques, our algorithmic proof, which uses local greedy fixes on a geometric grid, is fundamentally different, and arguably more elementary.

\section{Preliminaries and Proof Overview}
\label{sec:prelims-overview}

\subsection{Preliminaries}
\label{sec:prelims}

We frame the problem of sequential testing for Bernoulli trials in terms of analyzing the hitting times of biased random walks over a grid, where each policy is described by a set of stopping points. Namely, any policy $\set{P}$ can be thought of simply as a coloring of the infinite two-dimensional grid $\mathbb{Z}_{\ge 0}^2$ with three colors: red, blue and white. Every cell $(h,t)$ in the grid (for $h, t \in \mathbb{Z}_{\ge 0}$) represents a state where we have observed $h$ heads and $t$ tails. A white cell indicates that the policy defers the decision and elects to toss the coin again at that state. A red/blue colored cell represents a \textit{hitting point} of the policy: if the cell is colored blue, the policy stops and accepts $H_1: p=\frac{1}{2}+\eps$. If the cell is colored red, the policy stops and accepts $H_0:p=\frac{1}{2}-\eps$. Refer to \Cref{fig:policy} for an illustration of an arbitrary policy. In what follows, we will interchangeably identify a policy $\mcP$ with its coloring.

With this view, the expected number of coin tosses by the policy $\mcP$ is simply the expected \textit{hitting time} of a random walk that starts at $(0,0)$, where we move right with probability $p$ and down with probability $1-p$. We note that in this representation, a policy $\mcP$ only keeps track of the \textit{count} of heads observed at any point, and does not keep track of the precise sequence of the coin flip outcomes. This representation can be applied without loss of generality as the count of heads is a sufficient statistic for the bias of the coin. We constrain our analysis to policies with finite expected hitting time, since a policy with infinite expected hitting time cannot be optimal.

\begin{figure}[H]
    \centering
    \includegraphics[scale=0.7]{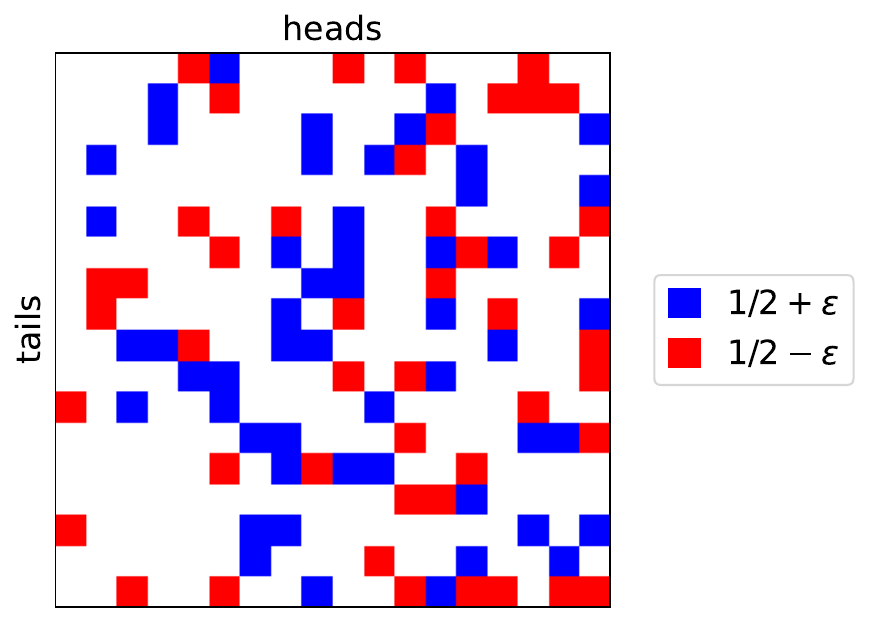}
    \caption{A policy $\set{P}$ truncated at a finite horizon. A point $(h,t)$ in the grid represents a state where we have observed $h$ heads and $t$ tails. The top left cell corresponds to the point $(0,0)$.}
    \label{fig:policy}
\end{figure}

The only technical tool we will require in our proof is the following consequence of the Martingale Stopping Theorem, also known as Doob's Optional Stopping Theorem, which characterizes the hitting time of a one-dimensional random walk. For a derivation of this theorem, we refer the reader to Theorem 3 in \cite{valiant23}.

\begin{theorem}[Hitting times for linear boundaries]
    \label{thm:linear-boundary}
    Suppose we start at the point $(h,t)=(0,0)$, and independently at every timestep, we either toss heads with probability $p$, and tails with probability $1-p$. Let $r=\frac{1}{p}-1$. Let the random variable $T$ denote the first time $h-t=b$ or $h-t=-a$, and $Z_T$ denote $h-t$ at this time. Then,
    \begin{align}
        &\Pr[Z_T=-a]=1-\Pr[Z_T=b] = \frac{1-r^{b}}{r^{-a}-r^{b}}, \label{eqn:linear-policy-probability} \\
        &\E[T] = \frac{b-(a+b)\left(\frac{1-r^{b}}{r^{-a}-r^{b}}\right)}{2p-1}. \label{eqn:linear-policy-hitting-time}
    \end{align} 
\end{theorem}

In our proof, we will use the fact that upon fixing $a+b=2c$, $\E[T]$ is bounded above by $c^2$; this can be verified from \eqref{eqn:linear-policy-hitting-time}.

\subsection{Proof Overview}
\label{sec:overview}

We first recall our objective in \Cref{thm:linear-policy-beta}: for any tradeoff parameter $\beta$, we want to show the existence of a linear policy $\mcP_c$---one which halts when $|h-t|=c$---having profile $(\delta_c, H_c)$, such any other policy $\mcP$ having profile $(\delta, H)$ satisfies
\begin{align*}
    (\delta^+_c + \delta^-_c) + \beta(H^+_c + H^-_c) \le (\delta^+ + \delta^-) + \beta(H^+ + H^-).
\end{align*}
For a policy $\mcP$ having profile $(\delta, H)$, let $R_\mcP(\beta)=(\delta^+ + \delta^-) + \beta(H^+ + H^-)$ denote its Bayes risk. We want to show that $R_{\mcP_c}(\beta) \le R_{\mcP}(\beta)$ for every policy $\mcP$.

To this end, fix a policy $\mcP$, and consider its representation as a coloring of the two-dimensional grid, as in \Cref{fig:policy}. In order to show that there exists $\mcP_c$ such that $R_{\mcP_c}(\beta) \le R_{\mcP}(\beta)$, we will show that, for a certain value of $c$, we can convert the geometric representation of $\mcP$, step-by-step, into the geometric representation of $\mcP_c$. That is, we can construct a finite series of policies $\mcQ_1, \mcQ_2,\dots,\mcQ_m$, where $\mcQ_1=\mcP$, $\mcQ_m=\mcP_c$, and each $\mcQ_{i+1}$ is a structured modification of the coloring given by $\mcQ_{i}$. Furthermore, we can ensure that $R_{\mcQ_1}(\beta) \le R_{\mcQ_{2}}(\beta)+\gamma$, and that $R_{\mcQ_{i+1}}(\beta) \le R_{\mcQ_i}(\beta)$ for every $i=2,3,\dots,m-1$, where $\gamma$ can be taken to be arbitrarily small. Chaining together the inequalities will then establish the result.

We will now elaborate on how this chain of intermediate policies is constructed. The procedure is illustrated in \Cref{fig:overview} below, and consists of 5 main steps, each of which is detailed in \Cref{sec:proof_new}.

\begin{figure}[t]
    \centering
    \begin{adjustbox}{center}
    \begin{tikzpicture}[baseline=(current bounding box.center)]
        \node[anchor=south west, inner sep=0] (img1) at (0,0) {
            \begin{minipage}{0.265\textwidth}
                \centering
                \includegraphics[width=\linewidth]{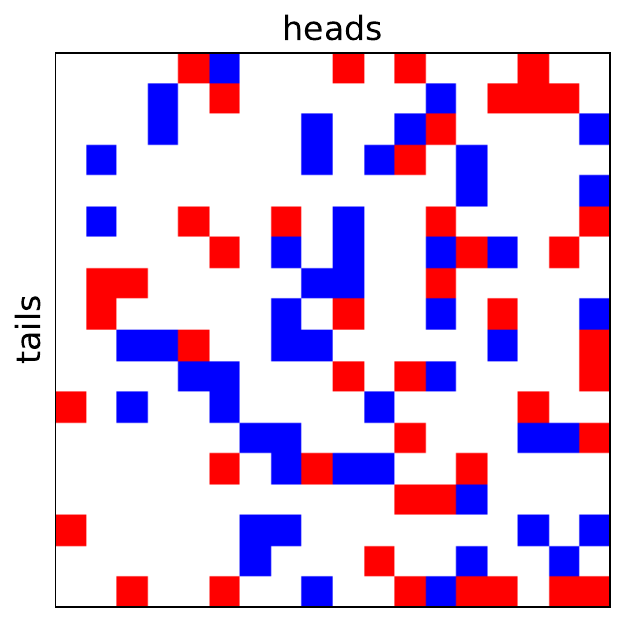}
            \end{minipage}
        };

        \node (arrow) at (5.2,2) {$\xRightarrow[\text{}]{\text{Step 1}}$};

        \node[anchor=south west, inner sep=0] (img2) at (6,0) {
            \begin{minipage}{0.265\textwidth}
                \centering
                \includegraphics[width=\linewidth]{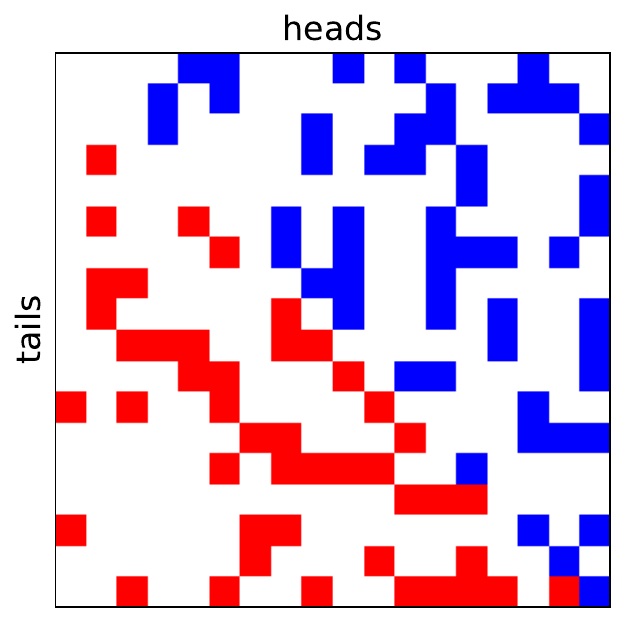}
            \end{minipage}
        };

        \node (arrow) at (11.2,2) {$\xRightarrow[]{\text{Step 2}}$};

        \node[anchor=south west, inner sep=0] (img2) at (12,0) {
            \begin{minipage}{0.265\textwidth}
                \centering
                \includegraphics[width=\linewidth]{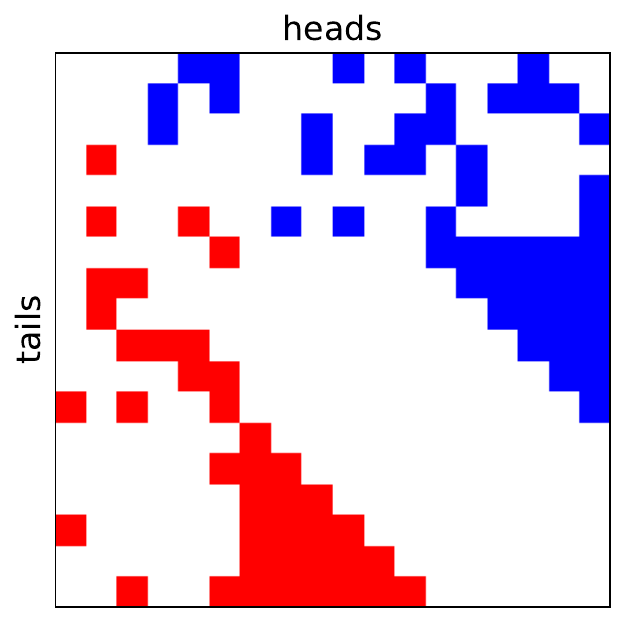}
            \end{minipage}
        };

        \node[anchor=south west, inner sep=0] (img1) at (0,-4.8) {
            \begin{minipage}{0.265\textwidth}
                \centering
                \includegraphics[width=\linewidth]{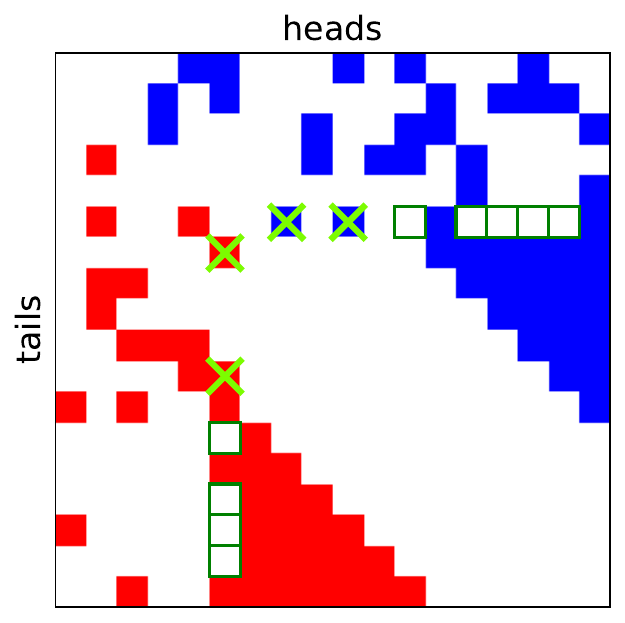}
            \end{minipage}
        };

        \node (arrow) at (5.2,-2.8) {$\xRightarrow[\text{}]{\text{Steps 3,4,5}}$};

        \node[anchor=south west, inner sep=0] (img1) at (6,-4.8) {
            \begin{minipage}{0.265\textwidth}
                \centering
                \includegraphics[width=\linewidth]{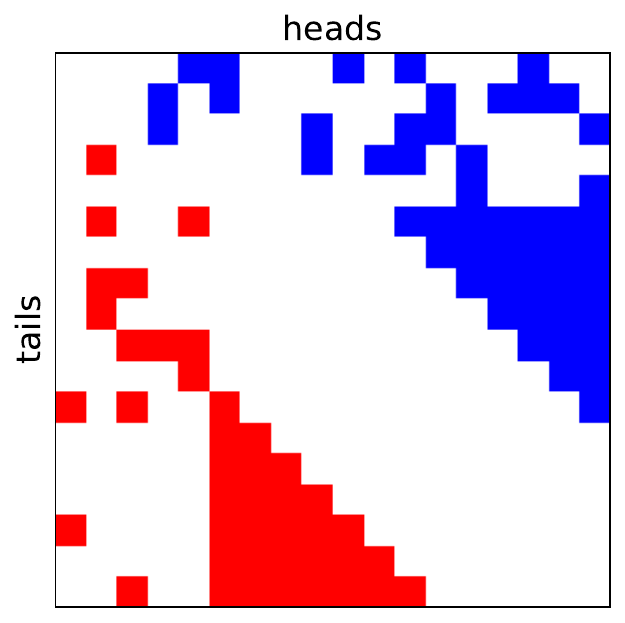}
            \end{minipage}
        };

        \node (arrow) at (11.2,-2.8) {$\xRightarrow[\text{repeatedly}]{\text{Steps 3,4,5}}$};

        \node[anchor=south west, inner sep=0] (img1) at (12,-4.8) {
            \begin{minipage}{0.265\textwidth}
                \centering
                \includegraphics[width=\linewidth]{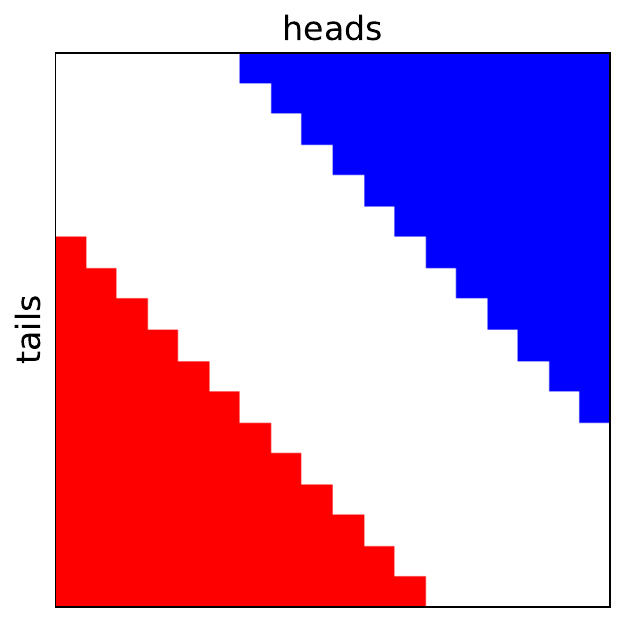}
            \end{minipage}
        };
        
    \end{tikzpicture}
    \end{adjustbox}
    \caption{An overview of our algorithmic proof, which systematically converts any arbitrary policy $\mcP$ to a linear policy $\mcP_c$, while ensuring that the Bayes risk does not increase. Step 1 makes the upper and lower triangles monochromatic. Step 2 truncates the policy to the linear policy $\mcP_c$ beyond a suitably large horizon. Steps 3,4,5 together extend the linear policy by one level, by adding the outlined cells, and erasing the crossed out cells; repeating these latter steps finitely many times results in the desired linear policy $\mcP_c$.}
    \label{fig:overview}
\end{figure}
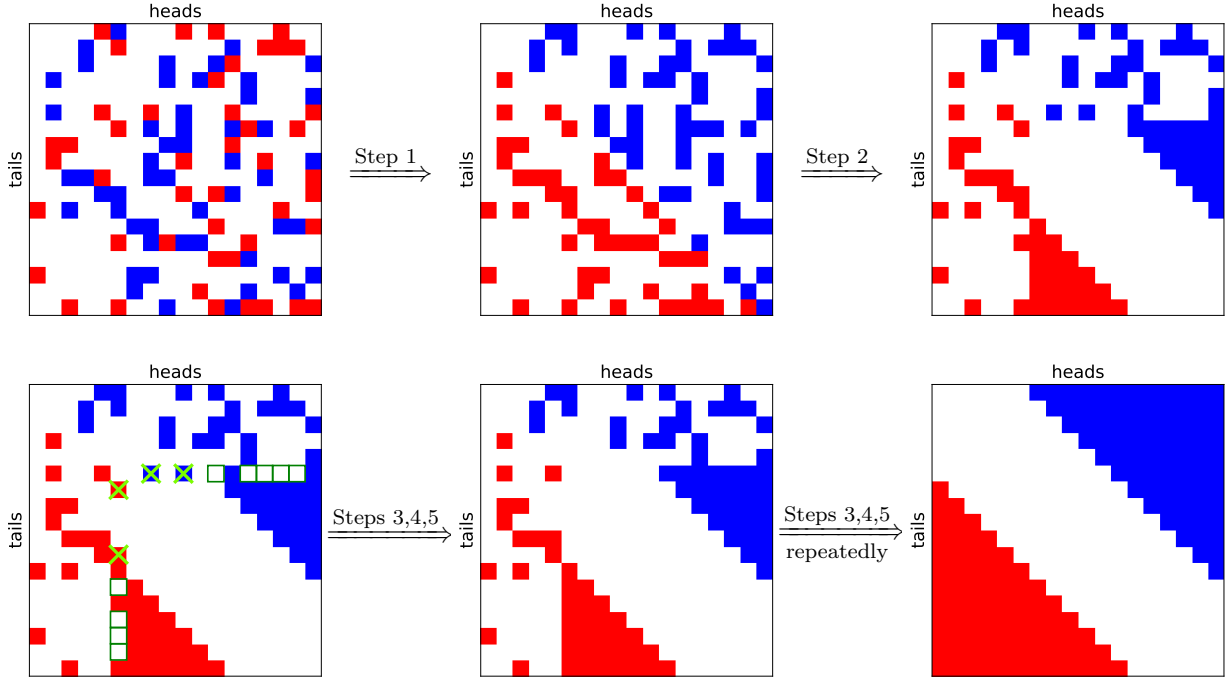

We begin by observing that cells in the policy that have not been colored according to the more frequent outcome at that cell can be recolored to match it at no cost. That is, any cell $(h,t)$ where $h\ge t$, and any cell $(h,t)$ where $h<t$, can be recolored to blue and red respectively, if this is not already the case in $\mcP$. This recoloring does not affect any hitting times. Furthermore, the sum of failure probabilities in the recolored policy can only decrease, essentially because the likelihood that a $1/2+\eps$ biased trajectory arrives at a point with $h \ge t$ is no less than this likelihood is for a $1/2-\eps$ trajectory (and vice versa for a point with $h<t$). Thus, Step 1 simply constitutes coloring the upper and lower triangles of the policy in accordance with the more frequent outcome.

As our Step 2, we truncate the recolored policy to be a linear policy beyond a suitably large finite horizon $n$. By choosing $n$ to be appropriately large, the hitting trajectories that get affected after the truncation are necessarily quite long. In particular, the total probability of these affected trajectories can be bounded above by a function $p(n)$, which goes to $0$ as $n$ increases. Moreover, because we truncate to a linear policy, we can also bound the expected hitting time of any trajectory that enters the channel between the parallel lines of the linear policy (which is of width $2c$), using \Cref{thm:linear-boundary}, by $O(c^2)$. In total, we only increase $R_\mcP(\beta)$ by a quantity that is $o_n(1)$; so, choosing $n$ large enough can make the increase be at most $\gamma$ for any desired $\gamma$. We note that this truncation step works for any value of $c$.

Our strategy hereafter is to pull the horizon of truncation from $n+1$ to 0, which will result in a fully linearized policy. We do this one step at a time---that is, we systematically extend the truncation horizon from $n+1$ to $n,n-1,\dots,0$. To this end, let us think about what it entails to extend the truncation horizon by one level from $n+1$ to $n$: we require coloring some points, and also erasing some extraneous hitting points at $n$ to be in accordance with how the linear policy should be (see \Cref{fig:overview}). Steps 3, 4 and 5 of our algorithmic procedure together achieve precisely this outcome, and constitute the bulk of our main technical arguments. In analyzing how adding and erasing points affects failure probabilities and hitting times, we again rely heavily on the property that trajectories that get modified necessarily pass through the channel of truncation, which is spanned on both sides by a pair of parallel lines. This crucially allows us to analyze the failure probabilities and hitting times of these trajectories, once they have entered this channel, using \Cref{thm:linear-boundary}. It is in the analysis of these steps that we require carefully choosing $c$ as a function of $\beta$, so that the addition and erasure of points which constitutes extending the truncation does not increase $R_\mcP(\beta)$. Our calculations show that there is always a valid choice of $c$ for this; indeed, the same $c$ works for extending the truncation at all the levels from $n$ to 0. Summarily, we can repeat Steps 3, 4, and 5 for a total of $n+1$ times in order to end up with the fully linearized policy, while ensuring that $R_\mcP(\beta)$ never increases. With this overview, we now move on to formalizing all these arguments in the next section.

\section{Proof of Optimality of the Linear Policy}
\label{sec:proof_new}

\subsection*{Step 1: Recoloring the upper and lower triangles}
\label{sec:step-1-fix-colors}

We first make a simple observation that a policy should, at any hitting point, predict the bias in the direction of the most frequent outcome. Geometrically (see \Cref{fig:step-1-fix-colors}), any policy can be improved at no cost by not having any red hitting points (i.e., declaring $1/2-\eps$) in the upper triangle where $h \ge t$, and any blue hitting points (i.e., declaring $1/2+\eps$) in the lower triangle where $t > h$. In particular, we can recolor every such point, without changing the hitting time of the policy, while decreasing the failure probability.

\begin{figure}[H]
    \centering
    \begin{adjustbox}{center}
    \begin{tikzpicture}[baseline=(current bounding box.center)]
        \node[anchor=south west, inner sep=0] (img1) at (0,0) {
            \begin{minipage}{0.35\textwidth}
                \centering
                \includegraphics[width=\linewidth]{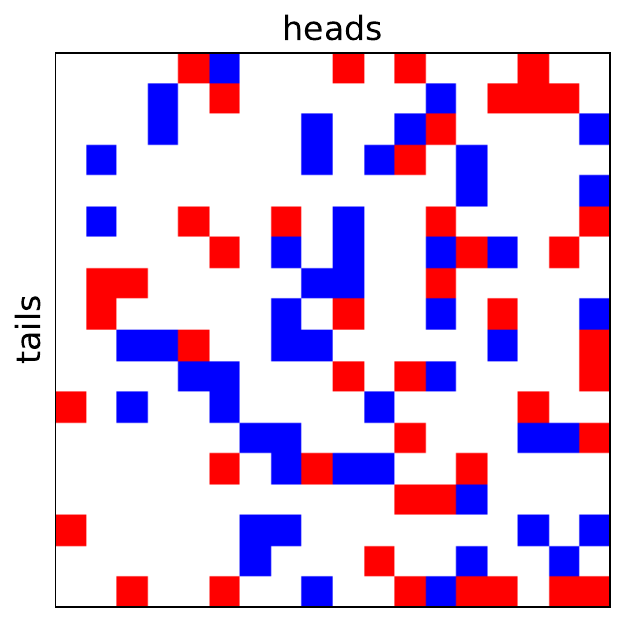}
            \end{minipage}
        };

        \node (arrow) at (7,3) {\Huge$\Longrightarrow$};

        \node[anchor=south west, inner sep=0] (img2) at (8,0) {
            \begin{minipage}{0.35\textwidth}
                \centering
                \includegraphics[width=\linewidth]{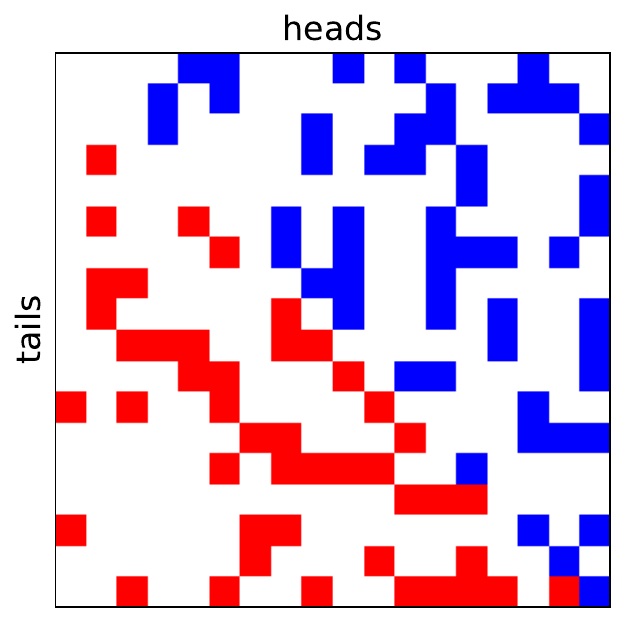}
            \end{minipage}
        };
    \end{tikzpicture}
    \end{adjustbox}
    \caption{Fixing colors of the hitting points in the upper and lower triangles of the policy.}
    \label{fig:step-1-fix-colors}
\end{figure}

\begin{claim}[Monochromatic Triangles]
    \label{claim:step-1-fix-colors}
    Let $\set{P}$ be a policy %
    with profile $(\delta_1, H_1)$. Let $\set{P}'$ be the policy that is otherwise identical to $\set{P}$, but colors every red hitting point $(h,t)$ in $\set{P}$ that satisfies $h\ge t$ as blue instead. 
    Let $(\delta_2, H_2)$ be the profile of $\set{P}'$. Then,
    \begin{align*}
        H_2 &= H_1 \\
        \delta^+_2+\delta^-_2 &\le \delta^+_1+\delta^-_1.
    \end{align*}
    A similar statement holds true for $\set{P}'$ that colors every blue hitting point $(h,t)$ in $\set{P}$ that satisfies $h < t$ as red instead.\footnote{Diagonal points can be colored either blue or red, since the probability that a trajectory arrives at a diagonal point is the same under either hypothesis. We color the diagonal points blue for concreteness.}
\end{claim}
\begin{proof}
    Note that every hitting point in $\set{P}$ continues to remain a hitting point in $\set{P}'$, which means that the probability of arriving at any hitting point is unchanged. Thus, $H^+_2=H^+_1$ and $H^-_2=H^-_1$.

    After recoloring all the red hitting points $(h,t)$ in $\set{P}$ that satisfy $h \ge t$ to blue in $\set{P}'$, note that only trajectories that stop at any such $(h, t)$ are affected. Since each such trajectory is now blue, it follows that $\delta^+_2 < \delta^+_1$ and that $\delta^-_2 >\delta^-_1$. Let $\Delta\delta^+=\delta^+_2-\delta^+_1$ and $\Delta\delta^-=\delta^-_2-\delta^-_1$. Let $S$ denote the set of all recolored hitting points, and let $\mathcal{T}_{(h,t)}$ denote the set of trajectories that arrive and stop at some recolored hitting point $(h,t) \in S$. Then, we have that
    \begin{align*}
        &\Delta\delta^+ = -\sum_{(h,t) \in S}\sum_{\tau \in \mathcal{T}_{(h,t)}}\Pr_{\frac12+\eps}[\tau] \quad \text{ and } \quad 
        \Delta\delta^- = \sum_{(h,t)\in S}\sum_{\tau \in \mathcal{T}_{(h,t)}}\Pr_{\frac12-\eps}[\tau],
    \end{align*}
    which means that
    \begin{align*}
        \Delta\delta^+ + \Delta\delta^- &= \sum_{(h,t) \in S}\sum_{\tau \in \mathcal{T}_{(h,t)}}\left[\Pr_{\frac12-\eps}[\tau]-\Pr_{\frac12+\eps}[\tau]\right] \\
        &= \sum_{(h,t) \in S}\sum_{\tau \in \mathcal{T}_{(h,t)}}\left[\left(\frac12-\eps\right)^{h}\left(\frac12+\eps\right)^{t}-\left(\frac12-\eps\right)^{t}\left(\frac12+\eps\right)^{h}\right] \\
        &= \sum_{(h,t) \in S}\sum_{\tau \in \mathcal{T}_{(h,t)}}\left[\left(\frac12-\eps\right)^{t}\left(\frac12+\eps\right)^{t}\left[ \left(\frac12-\eps\right)^{h-t}-\left(\frac12+\eps\right)^{h-t}\right]\right] \le 0,
    \end{align*}
    where the last inequality follows because $h \ge t$.
\end{proof}

\subsection*{Step 2: Truncating to a linear policy}
\label{sec:step-2-truncate-to-linear}

In order to constrain our analysis to a finite number of steps, we need to contain the differences between an arbitrary policy and the target linear policy to a finite region. We do so by \emph{truncating} policies after a far-enough point to match the target linear policy, with arbitrarily small impact on the Bayes risk. For a sufficiently large horizon, the probability of any trajectory being altered is exponentially small, and almost all trajectories are identical among both policies. An illustration of such a truncation is in \Cref{fig:step-2-truncate-to-linear}.

\begin{definition}[Truncated Linear Policy]
    \label{def:truncated-linear-policy}
    Given a policy $\set{P}$, the truncated linear policy $\set{P}_{c,n}$ (for $c, n \ge 1$) mirrors the linear policy $\set{P}_c$ in the region $h, t \geq n$. All other points are colored according to $\P$.
\end{definition}

\begin{figure}[H]
    \centering
    \begin{adjustbox}{center}
    \begin{tikzpicture}[baseline=(current bounding box.center)]
        \node[anchor=south west, inner sep=0] (img1) at (0,0) {
            \begin{minipage}{0.35\textwidth}
                \centering
                \includegraphics[width=\linewidth]{img/policy-diagonal-fixed.pdf}
            \end{minipage}
        };

        \node (arrow) at (7,3) {\Huge$\Longrightarrow$};

        \node[anchor=south west, inner sep=0] (img2) at (8,0) {
            \begin{minipage}{0.35\textwidth}
                \centering
                \includegraphics[width=\linewidth]{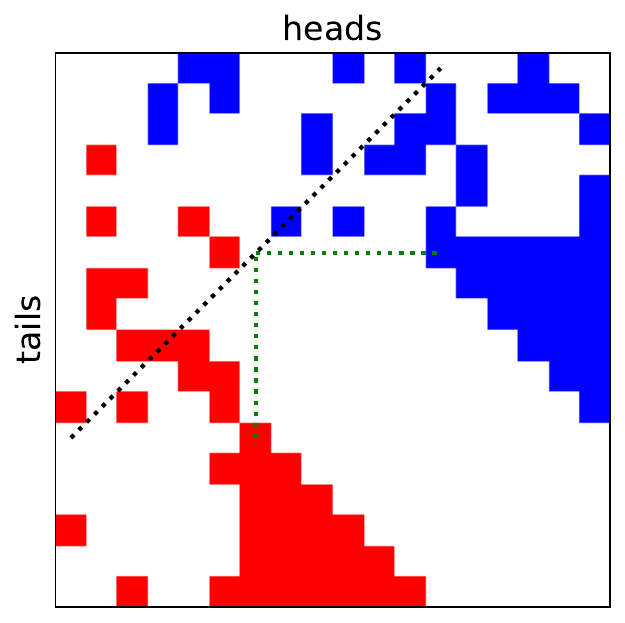}
            \end{minipage}
        };
    \end{tikzpicture}
    \end{adjustbox}
    \caption{Truncating a policy to a linear policy. Observe that any trajectory which stops before intersecting the dotted black line is unaffected. Only trajectories which intersect the green dotted lines may stop later; all other trajectories are unaffected or trimmed.}
    \label{fig:step-2-truncate-to-linear}
\end{figure}

\begin{claim}[Truncation to Linear Policy]
    \label{claim:step-2-truncation-to-linear}
    Let $\set{P}$ be a recolored policy having profile $(\delta, H)$.
    Fix $c \in \mathbb{Z}_{\ge 1}$ and  $\beta \ge 0$. For any $\gamma > 0$, there exists a sufficiently large $n$ such that the truncated linear policy $\set{P}_{c,n+1}$, having profile $(\delta_{c,n+1}, H_{c, n+1})$
    satisfies
    \begin{align}
        (\delta^+_{c,n+1} + \delta^-_{c,n+1}) + \beta(H^+_{c,n+1} + H^-_{c,n+1}) \le (\delta^+ + \delta^-) + \beta(H^+ + H^-) + \gamma.
    \end{align}
\end{claim}
\begin{proof}

    For any $n$, let $p(n)$ denote the probability of the policy $\set{P}$ realizing a trajectory that tosses at least $n$ times, i.e.,
    \begin{align*}
        \label{def:p(n)}
        p(n) := \max\left(\Pr_{\half+\eps}\left[\text{trajectory with $\ge n$ tosses}\right], \Pr_{\half-\eps}\left[\text{trajectory with $\ge n$ tosses}\right]\right).
    \end{align*}
    In other words, $p(n)$ sums up the probabilities of all the trajectories that pass through the line $h+t=n$. Observe that $p(n)$ must satisfy $\lim_{n \to \infty}p(n)=0$; otherwise, at least one of $H^+, H^-$ will be infinite.
    
    Now suppose we truncate $\set{P}$ starting at position $n+1$ to the linear policy with intercept $c$. Then, observe that all trajectories which are modified \textit{must} pass through the line $h+t=2(n+1)$ (this is the dotted black line in \Cref{fig:step-2-truncate-to-linear}). 
    That is, these must realize at least $2n+2$ tosses. The probability of realizing such trajectories is, by definition, at most $p(2n+2)$. Consequently, each of the failure probabilities in $\delta_{c,n+1}$ are each increased by at most $p(2n+2)$.

    Next, the only hitting trajectories that can \textit{increase} in their hitting time due to the truncation \emph{must} intersect the set of points $\{(\text{heads}=n+1, \text{tails}=n+1+i)\}$ or $\{(\text{heads}=n+1+i, \text{tails}=n+1)\}$ for $0 \le i \le c$ (i.e., the green dotted segments in \Cref{fig:step-2-truncate-to-linear}). All other trajectories not passing through this set are either trimmed or unaffected in the truncated policy.
    
    Conditioned on arriving at any point in this set, the expected additional time before stopping in the truncated linear policy is at most $c^2$ by \eqref{eqn:linear-policy-hitting-time}. Thus, for trajectories which intersect this set, the expected hitting time may increase at most by $c^2$. Since the probability of intersecting this set is again at most $p(2n+2)$, we can conclude that
    \begin{align*}
        (\delta^+_{c,n+1} + \delta^-_{c,n+1}) + \beta(H^+_{c,n+1} + H^-_{c,n+1}) &\le (\delta^+ + \delta^- + 2p(2n+2)) + \beta(H^+ + H^- + 2c^2 p(2n+2)) \\
        &= (\delta^+ + \delta^-) + \beta(H^+ + H^-) + 2p(2n+2)(1+\beta c^2).
    \end{align*}
    Finally, since $\lim_{n \to \infty}p(n)=0$, there exists $n$ large enough such that $2p(2n+2)(1+\beta c^2) \le \gamma$.
\end{proof}

\subsection*{Step 3: Completing the truncation boundary}
\label{sec:step-3-fill-gaps}

After truncating the policy to a linear policy, our next step consists of filling in the next layer of the linear boundary. Namely, there are a number of white cells around the truncation boundary which can be converted to red or blue hitting points at no cost (refer to the cells outlined in green in the left part of \Cref{fig:step-3-fill-gaps}). This is because any trajectory arriving at such a white cell goes on to stop at the corresponding color with probability 1.

\begin{figure}[H]
    \centering
    \begin{adjustbox}{center}
    \begin{tikzpicture}[baseline=(current bounding box.center)]
        \node[anchor=south west, inner sep=0] (img1) at (0,0) {
            \begin{minipage}{0.35\textwidth}
                \centering
                \includegraphics[width=\linewidth]{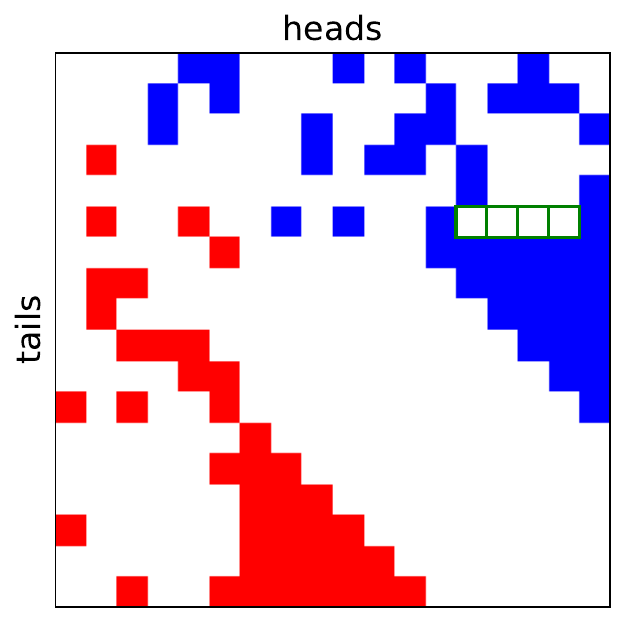}
            \end{minipage}
        };

        \node (arrow) at (7,3) {\Huge$\Longrightarrow$};

        \node[anchor=south west, inner sep=0] (img2) at (8,0) {
            \begin{minipage}{0.35\textwidth}
                \centering
                \includegraphics[width=\linewidth]{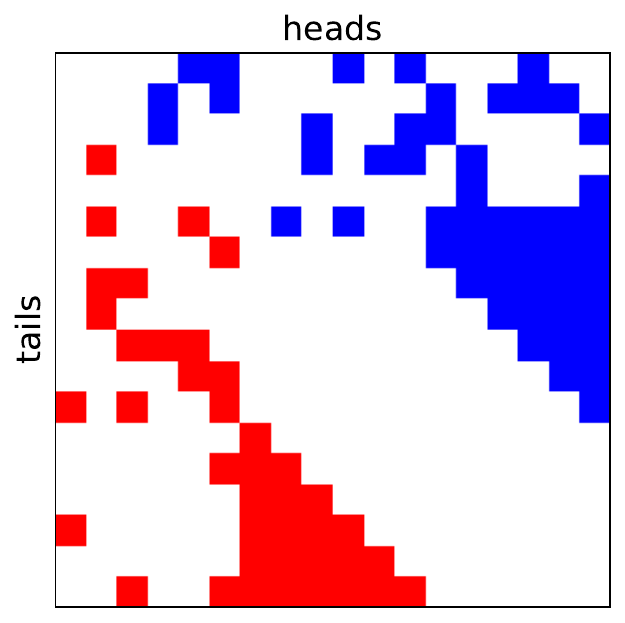}
            \end{minipage}
        };
    \end{tikzpicture}
    \end{adjustbox}
    \caption{Filling in gaps at the next layer of the truncation boundary. Observe that a trajectory arriving at any of the cells outlined in green will go on to hit a blue cell with probability 1.}
    \label{fig:step-3-fill-gaps}
\end{figure}

\begin{claim}[Filling the Next Layer]
    \label{claim:step-3-fill-gaps}
    Suppose that for $c \ge 1$, $\set{P}_{c,n+1}$ is a truncated linear policy with profile $(\delta_1, H_1)$.
    \begin{enumerate}
        \item[(1)] Color every white cell $(h,t)$ in $\set{P}_{c,n+1}$, where $t=n$ and $h > t+c$ blue.
        \item[(2)] Color every white cell $(h,t)$ in $\set{P}_{c,n+1}$, where $h=n$ and $t > h+c$ red. 
    \end{enumerate}
    In either case (1) or (2), let $(\delta_2, H_2)$ be the profile of the modified policy. We have that
    \begin{align*}
        H^+_2 &\le H^+_1, \quad H^-_2 \le H^-_1 \\
        \delta^+_2&=\delta^+_1, \qquad  \delta^-_2=\delta^-_1.
    \end{align*}
\end{claim}
\begin{proof}
    Consider case (1). For any white cell $(h,t)$ in $\set{P}_{c,n+1}$ that satisfies $t=n$ and $h > t+c$, it holds that any hitting point to its right at the same level is colored blue by the recoloring in Step 1. Furthermore, all points $(h', t+1)$ for $h' \ge h$ are also colored blue by the truncation at $n+1$. This implies that any trajectory arriving at this white cell will stop at a blue hitting point with probability 1. Thus, coloring this white cell blue does not affect any failure probabilities. Moreover, if there is a positive probability of arriving at this white cell, coloring the cell strictly decreases the expected hitting time. The analysis for case (2) is identical.
\end{proof}

\subsection*{Step 4: Extending the linear policy}
\label{label:step-4-place-hitting-point}

After we have filled in the gaps around the truncated linear policy, we are one step closer to extending the linear policy by another level. However, while hitting points at all $(h,t)$ where $t=n$ and $h > t+c$ have been colored blue, we must also ensure that there is a blue hitting point precisely at $(h,t)$ where $t=n, h=t+c$ (and red analogously). Our next technical lemma shows that placing precisely this hitting point can only result in improvement.

\begin{figure}[H]
    \centering
    \begin{adjustbox}{center}
    \begin{tikzpicture}[baseline=(current bounding box.center)]
        \node[anchor=south west, inner sep=0] (img1) at (0,0) {
            \begin{minipage}{0.35\textwidth}
                \centering
                \includegraphics[width=\linewidth]{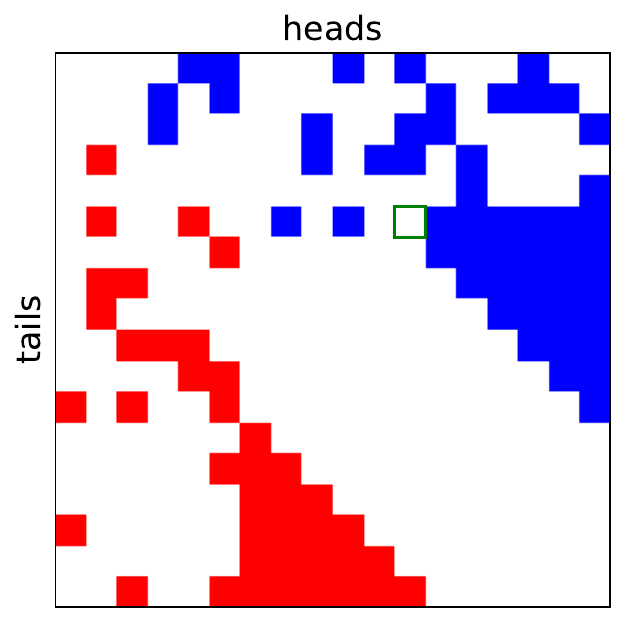}
            \end{minipage}
        };

        \node (arrow) at (7,3) {\Huge$\Longrightarrow$};

        \node[anchor=south west, inner sep=0] (img2) at (8,0) {
            \begin{minipage}{0.35\textwidth}
                \centering
                \includegraphics[width=\linewidth]{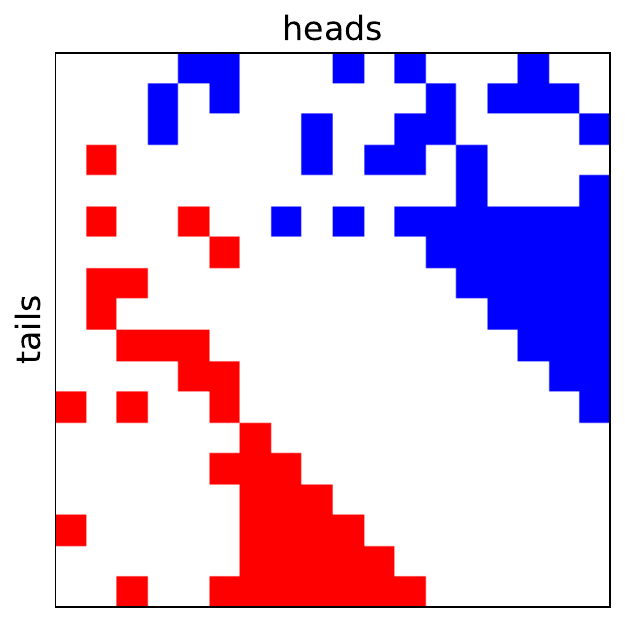}
            \end{minipage}
        };
    \end{tikzpicture}
    \end{adjustbox}
    \caption{Placing an additional hitting point at the right location (outlined green cell) extends the next level of the truncation.}
    \label{fig:step-4-place-hitting-point}
\end{figure}

\begin{lemma}[Placing the Next Hitting Point]
    \label{lemma:step-4-place-hitting-point}
    Suppose that for $c \ge 1$, $\set{P}_{c,n+1}$ is a truncated linear policy with profile $(\delta_1, H_1)$.
    \begin{enumerate}
        \item[(1)] Suppose that for $t=n$, there is a blue hitting point at every $h$ satisfying $h > t+c$.\footnote{That is, all gaps around the truncation have been filled, as described in the previous Step 3.} Now consider placing an additional blue hitting point at $h=t+c$, if there isn't already one. 
        \item[(2)] Suppose that for $h=n$, there is a red hitting point at every $t$ satisfying $t > h+c$. Now consider placing an additional red hitting point at $t=h+c$, if there isn't already one. 
    \end{enumerate}
    In either case (1) or (2), let $(\delta_2, H_2)$ be the profile of the modified policy. We have that
    \begin{equation}
        (\delta^+_2 + \delta^-_2) + \beta(H^+_2 + H^-_2) \le (\delta^+_1 + \delta^-_1) + \beta(H^+_1 + H^-_1),
    \end{equation}
    for any $\beta=\beta(c, \eps)$ satisfying
    \begin{equation}
        \label{eqn:beta-lb}
        \beta \ge \frac{2\eps\alpha^c(\alpha-1)}{(\alpha^{c+1}-1)(\alpha^c+1)+2c\alpha^c(\alpha-1)}.
    \end{equation}
    Here, $\alpha=\frac{1+2\eps}{1-2\eps}$. Note that the lower bound on $\beta$ is independent of $n$.
\end{lemma}
\begin{proof}
    For concreteness, we consider case (1); the analysis for case (2) is identical due to symmetry. We suppose that the policy has undergone Steps 1-3. If the point $x:=(t+c, t)=(n+c, n)$ is already blue, we are done. Otherwise, $x$ was a white cell, and we recolor it blue. 
    
    Observe that trajectories which do not intersect $x$ are unaffected. Before coloring $x$ blue, any trajectory that arrived at $x$ could go on to hit either the $\half+\eps$ or the $\half-\eps$ boundary. After making $x$ a blue hitting point, this trajectory halts and declares $\half+\eps$. By making it more likely for a trajectory to halt at a blue hitting point and declare $\+$, we \emph{decrease} the failure probability $\delta^+$; conversely, we \textit{increase} the failure probability $\delta^-$ if the coin were biased $\half-\eps$.

    Consider a trajectory $\tau$ that arrives at $x$ but then goes on to hit the $\half-\eps$ boundary in the original policy $\P_{c, n+1}$. Such a trajectory must necessarily flip tails next (otherwise, it hits the truncated linear boundary). At this point $(n+c, n+1)$, note that $\tau$ is at distance 1 and $2c-1$ from a pair of \textit{linear} boundaries!  %
    \begingroup
    \allowdisplaybreaks
    \begin{align*}
        \Delta \delta^- = \delta^-_2 - \delta^-_1 &= \Pr_{\half-\eps}\left[\text{coin arrives at }x\right]\cdot \left( \half+\eps\right) \cdot \Pr_{\half-\eps}\left[\text{down $2c-1$ tails before up $1$ head}\right] \\
        \Delta \delta^+ = \delta^+_2 - \delta^+_1 &= -\Pr_{\half+\eps}\left[\text{coin arrives at }x\right]\cdot \left( \half-\eps\right) \cdot \Pr_{\half+\eps}\left[\text{down $2c-1$ tails before up $1$ head}\right] \\
        &= -\alpha^c\cdot\Pr_{\half-\eps}\left[\text{coin arrives at }x\right]\cdot \left(\half-\eps\right) \cdot \Pr_{\half+\eps}\left[\text{down $2c-1$ tails before up $1$ head}\right].
    \end{align*}
    The last step follows from a simple algebraic transformation:
    \begin{align}
        \Pr_{\half-\eps}[\text{coin arrives at }x] &= \sum_{\text{trajectory $\tau$ arriving at $x$}}\left(\half-\eps\right)^{t+c}\left(\half+\eps\right)^{t} \nonumber \\
        &= \left(\frac{\half-\eps}{\half+\eps}\right)^{c}\sum_{\text{trajectory $\tau$ arriving at $x$}}\left(\half+\eps\right)^{t+c}\left(\half-\eps\right)^{t} \nonumber \\
        &= \alpha^{-c} \cdot \Pr_{\half+\eps}[\text{coin arrives at }x]. \label{eqn:calc-relating-plus-minus}
    \end{align}
    In the above, the event ``down $a$ tails before up $b$ heads'' corresponds precisely to the event that a trajectory arrives at a point $(h,t)$ satisfying $h-t=-a$ before it arrives at any point $(h,t)$ satisfying $h-t=b$. Therefore, we can directly apply the bounds from \Cref{thm:linear-boundary} to obtain:
    \begin{align*}
        \Delta \delta^+ + \Delta \delta^- &= \Pr_{\half-\eps}\left[\text{coin arrives at }x\right] \cdot \left[ \left( \half + \eps\right)\left(\frac{\alpha^{2c}-\alpha^{2c-1}}{\alpha^{2c}-1}\right) -\left(\half-\eps\right)\alpha^c\left(\frac{\alpha-1}{\alpha^{2c}-1}\right)\right] \\
        &=  \Pr_{\half-\eps}\left[\text{coin arrives at }x\right] \cdot \left(\frac{\alpha-1}{\alpha^{2c}-1}\right) \left[\left( \half + \eps\right)\alpha^{2c-1} -\left(\half-\eps\right)\alpha^c\right] > 0.
    \end{align*}
    \endgroup
    The expected hitting time decreases in both cases, as
    \begingroup
    \allowdisplaybreaks
    \begin{align*}
        \Delta H^{-} &= H^-_2-H^-_1 \\
        &= -\Pr_{\half-\eps}\left[\text{coin arrives at }x\right]\left[\left(\half-\eps\right) + \left(\half+\eps\right)\left(1+\E_{\half-\eps}\left[\text{down $2c-1$ tails or up $1$ head}\right]\right)\right] \\
        &= -\Pr_{\half-\eps}\left[\text{coin arrives at }x\right] \left[1 + \left(\half+\eps\right)\E_{\half-\eps}\left[\text{down $2c-1$ tails or up $1$ head}\right]\right] \\
        \Delta H^{+} &= H^+_2-H^+_1 \\
        &= -\Pr_{\half+\eps}\left[\text{coin arrives at }x\right] \left[\left(\half+\eps\right) + \left(\half-\eps\right)\left(1+\E_{\half+\eps}\left[\text{down $2c-1$ tails or up $1$ head}\right]\right)\right] \\
        &= -\alpha^c\cdot\Pr_{\half-\eps}\left[\text{coin arrives at }x\right] \left[\left(\half+\eps\right) + \left(\half-\eps\right)\left(1+\E_{\half+\eps}\left[\text{down $2c-1$ tails or up $1$ head}\right]\right)\right] \\
        &= -\alpha^c\cdot\Pr_{\half-\eps}\left[\text{coin arrives at }x\right] \left[1 + \left(\half-\eps\right)\E_{\half+\eps}\left[\text{down $2c-1$ tails or up $1$ head}\right]\right].
    \end{align*}
    \endgroup

    Plugging in the expressions from \Cref{eqn:linear-policy-hitting-time}, this gives
    \begingroup
    \allowdisplaybreaks
    \begin{align*}
        \Delta H^{+} + \Delta H^{-} &= -\Pr_{\half-\eps}\left[\text{coin arrives at }x\right] \\ 
        &\cdot \left[1 + \frac{(2\eps+1)(\alpha-\alpha^{2c+1}+2(\alpha-1)a^{2c}c)}{4\eps(\alpha^{2c+1}-\alpha)} + \alpha^c\left(1-\frac{(2\eps-1)(\alpha^{2c}-2\alpha c+2c - 1)}{4\eps(\alpha^{2c}-1)}\right)\right] \\
        &= -\Pr_{\half-\eps}\left[\text{coin arrives at }x\right]
        \left( \frac{\alpha^{2c+1}+(\alpha-1)(2c+1)\alpha^c - 1}{(\alpha-1)(\alpha^c+1)}\right) \le 0.
    \end{align*}
    \endgroup

    Now, consider first the case that $\Pr_{\half-\eps}\left[\text{coin arrives at }x\right]=0$, meaning there is no trajectory that can arrive at $x$. In this case, $\Delta \delta^{+} + \Delta \delta^{-} = \Delta H^+ + \Delta H^{-} = 0$, and so, the conclusion is true.

    Otherwise, when we take the ratio $\frac{\Delta \delta^{+} + \Delta \delta^{-}}{\Delta H^{+} + \Delta H^{-}}$, observe that the common factor $\Pr_{\half-\eps}\left[\text{coin arrives at }x\right]$ cancels out, and we obtain
    \begingroup
    \allowdisplaybreaks
    \begin{align*}
        \frac{\Delta \delta^+ + \Delta \delta^-}{\Delta H^+ + \Delta H^{-}} &= (\alpha-1)(\alpha^c+1)\frac{\left(\half-\eps\right)\alpha^c\left(\frac{\alpha-1}{\alpha^{2c}-1}\right)-\left( \half + \eps\right)\left(\frac{\alpha^{2c}-\alpha^{2c-1}}{\alpha^{2c}-1}\right)}{\alpha^{2c+1}+(\alpha-1)(2c+1)\alpha^c - 1} \\
        &= \frac{2\eps\alpha^c(1-\alpha)}{(\alpha^{c+1}-1)(\alpha^c+1)+2c\alpha^c(\alpha-1)}.
    \end{align*}
    \endgroup
    For the action to be an improvement, we would want
    \begingroup
    \allowdisplaybreaks
    \begin{align*}
        &\beta \cdot (\Delta H^{+} + \Delta H^{-}) \le -(\Delta \delta^{+} + \Delta \delta^{-}) \\
        \implies \qquad & \beta \ge -\frac{\Delta \delta^{+} + \Delta \delta^{-}}{\Delta H^{+} + \Delta H^{-}} \qquad (\text{since $\Delta H^{+} + \Delta H^{-} \le 0$}) \\
        &= \frac{2\eps\alpha^c(\alpha-1)}{(\alpha^{c+1}-1)(\alpha^c+1)+2c\alpha^c(\alpha-1)},
    \end{align*}
    \endgroup
    which is the claimed bound on $\beta$. The analysis is symmetric for case (2), concluding the proof.
\end{proof}

\subsection*{Step 5: Erasing extraneous hitting points}
\label{sec:step-5-erase-extra-points}

All that remains at this point to complete the truncation to a linear policy at $t=n$ is to remove any extraneous hitting points at this level. That is, we recolor all points between $(n, n)$ and $(n+c-1, n)$ (including $(n,n)$) white (these are the points to the left of the threshold placed in Step 4). Note that by Step 1, all such points are guaranteed to be blue. %
Similarly, we recolor all red points between $(n, n)$ and $(n, n+c-1)$ white. %

We will erase extraneous blue hitting points one-by-one in decreasing order of their $h$ values. That is, we erase hitting points one at a time as we move left from the point $(n+c,n)$ toward (and including) the point $(n,n)$ (see \Cref{fig:step-5-erase-extra-points}). Similarly, if we are working in the lower triangle, we erase hitting points one at a time as we move up from the point $(n, n+c)$ toward the diagonal point $(n,n)$. We will show that each such erasure results in a policy that is no worse. After all extraneous points are thus erased (in both the upper and lower triangles) we will notably have extended the truncated policy $\set{P}_{c, n+1}$ to $\set{P}_{c, n}$! 

\begin{figure}[H]
    \centering
    \begin{adjustbox}{center}
    \begin{tikzpicture}[baseline=(current bounding box.center)]
        \node[anchor=south west, inner sep=0] (img1) at (0,0) {
            \begin{minipage}{0.3\textwidth}
                \centering
                \includegraphics[width=\linewidth]{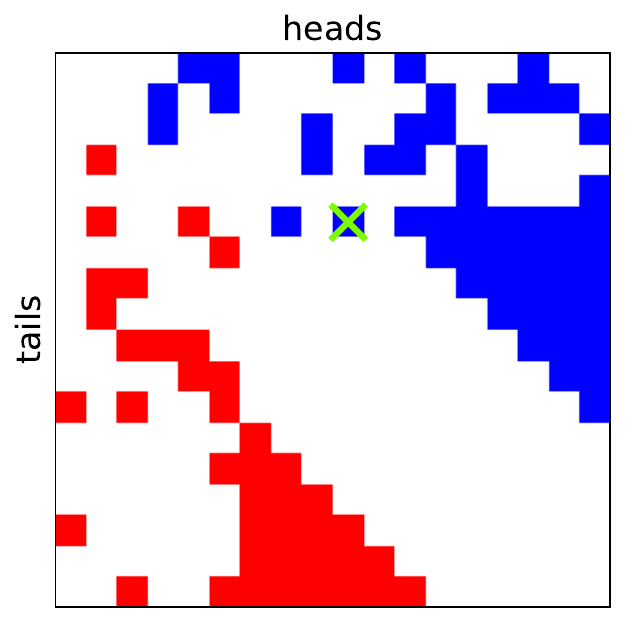}
            \end{minipage}
        };

        \node (arrow) at (5.5,2.5) {\Huge$\Rightarrow$};

        \node[anchor=south west, inner sep=0] (img2) at (6,0) {
            \begin{minipage}{0.3\textwidth}
                \centering
                \includegraphics[width=\linewidth]{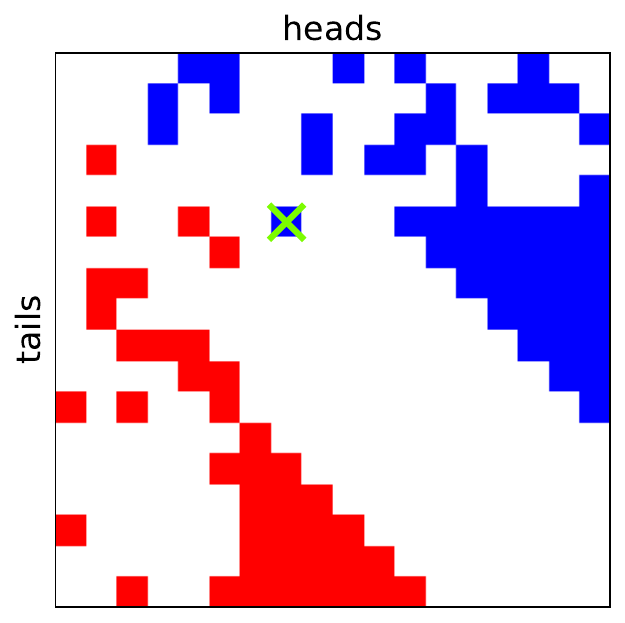}
            \end{minipage}
        };

        \node (arrow) at (11.5,2.5) {\Huge$\Rightarrow$};

        \node[anchor=south west, inner sep=0] (img2) at (12,0) {
            \begin{minipage}{0.3\textwidth}
                \centering
                \includegraphics[width=\linewidth]{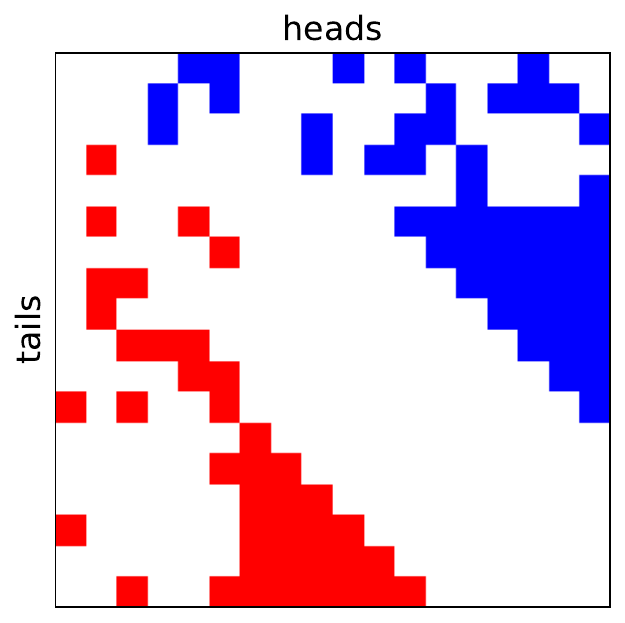}
            \end{minipage}
        };
    \end{tikzpicture}
    \end{adjustbox}
    \caption{Erasing extraneous points to the left of the linear policy threshold. Each erasure can only cause improvement.}
    \label{fig:step-5-erase-extra-points}
\end{figure}

\begin{lemma}[Erase Extra Hitting Points]
    \label{lemma:step-5-erase-extra-points}
    Suppose that for $c \ge 1$, $\set{P}_{c,n+1}$ is a truncated linear policy with profile $(\delta_1, H_1)$. 
    \begin{enumerate}
        \item[(1)] Suppose that for $t=n$, at every $h$ satisfying $h \ge t+c$, there is a blue hitting point.\footnote{In other words, Step 4 has been completed.} Among hitting points $(h,t)$ where $n \le h < n+c$, let $(h',t)$ be the one with the largest $h$ value. Now consider erasing $(h',t)$, i.e., coloring it white. %
        \item[(2)] Suppose that for $h=n$, at every $t$ satisfying $t \ge h+c$, there is a red hitting point. Among hitting points $(h,t)$ where $n < t < n+c$, let $(h,t')$ be the one with the largest $t$ value. Now consider erasing $(h,t')$, i.e., coloring it white. %
    \end{enumerate}
    In either case (1) or (2), let $(\delta_2, H_2)$ be the profile of the modified policy. We have that
    \begin{equation}
        (\delta^+_2 + \delta^-_2) + \beta(H^+_2 + H^-_2) \le (\delta^+_1 + \delta^-_1) + \beta(H^+_1 + H^-_1),
    \end{equation}
    for any $\beta=\beta(c, \eps)$ satisfying
    \begin{equation}
        \label{eqn:beta-ub}
         \beta \le \frac{2\eps\alpha^c(\alpha-1)}{(\alpha^c-\alpha)(\alpha^c+1)+2c\alpha^c(\alpha-1)}.
    \end{equation}
    Here, $\alpha=\frac{1+2\eps}{1-2\eps}$. Note that the upper bound on $\beta$ is independent of $n$. %
\end{lemma}
\begin{proof}
    Again, for concreteness, we consider case (1); the analysis for case (2) is identical due to symmetry. We assume that there exists some hitting point $x:=(h',t)$ to erase, or else we are done. Let $h'-t=c'$, where $0 \le c' < c$.

    The first main observation is that after erasing $x$, all hitting trajectories that do not pass through $x$ remain unaffected. Before erasing $x$, any trajectory arriving at $x$ immediately halts and declares $\half+\eps$.\footnote{Recall that we chose to color the diagonal points blue as well.} %
    After the move, this trajectory can now go on to hit either the $\half+\eps$ boundary, or the $\half-\eps$ boundary, making it more likely for an arbitrary trajectory to hit the $\half-\eps$ boundary. If the coin were biased $\half-\eps$, this trajectory \textit{decreases} the failure probability, whereas if the coin were biased $\half+\eps$, this trajectory now \textit{increases} the failure probability. The expected hitting time, however, \textit{increases} in both cases. 
    
    The second main observation is that after reaching $x$, all future hitting points lie on a linear policy. This is because the policy $\set{P}_{c,n+1}$ we are considering is a truncated linear policy. Hence, given that a trajectory arrives at $x$, we can analyze whether it hits the $\frac{1}{2} + \eps$ or $\frac{1}{2} - \eps$ boundary first by plugging in the formulae from \Cref{thm:linear-boundary} again.
    \begingroup
    \allowdisplaybreaks
    \begin{align*}
        \Delta \delta^{-} = \delta^-_2-\delta^-_1  &= -\Pr_{\half-\eps}\left[\text{coin arrives at }x\right]\cdot \Pr_{\half-\eps}\left[\text{down $c+c'$ tails before up $c-c'$ heads}\right] \\
        &= -\alpha^{-c'}\cdot\Pr_{\half+\eps}\left[\text{coin arrives at }x\right]\cdot \Pr_{\half-\eps}\left[\text{down $c+c'$ tails before up $c-c'$ heads}\right]\\
        \Delta \delta^{+} = \delta^+_2-\delta^+_1 &= +\Pr_{\half+\eps}\left[\text{coin arrives at }x\right]\cdot \Pr_{\half+\eps}\left[\text{down $c+c'$ tails before up $c-c'$ heads}\right] \\
        \Delta \delta^{+} + \Delta \delta^{-} &= \Pr_{\half + \eps}\left[\text{coin arrives at }x\right]\cdot\left[\left(\frac{\alpha^{c-c'}-1}{\alpha^{2c}-1}\right) -\alpha^{-c'}\left(\frac{\alpha^{2c}-\alpha^{c+c'}}{\alpha^{2c}-1}\right)\right] \\
        &= \Pr_{\half+\eps}\left[\text{coin arrives at }x\right]\cdot\frac{(1-\alpha^{c-c'})(\alpha^c-1)}{(\alpha^{2c}-1)} \\
        &= \Pr_{\half+\eps}\left[\text{coin arrives at }x\right]\cdot\frac{1-\alpha^{c-c'}}{\alpha^{c}+1} \le 0.
    \end{align*}
    \endgroup
    In the above, we used a calculation similar to \eqref{eqn:calc-relating-plus-minus} relating $\Pr_{\half+\eps}[\cdot]$ and $\Pr_{\half-\eps}[\cdot]$.
    The change in the expected hitting times is
    \begingroup
    \allowdisplaybreaks
    \begin{align*}
        \Delta H^{-} = H^-_2-H^-_1
        &= \Pr_{\half-\eps}\left[\text{coin arrives at }x\right]\cdot \E_{\half-\eps}\left[\text{down $c+c'$ tails or up $c-c'$ heads}\right] \\
        &= \alpha^{-c'}\cdot\Pr_{\half+\eps}\left[\text{coin arrives at }x\right] \cdot \E_{\half-\eps}\left[\text{down $c+c'$ tails or up $c-c'$ heads}\right] \\
        \Delta H^{+} = H^+_2-H^+_1
        &= \Pr_{\half+\eps}\left[\text{coin arrives at }x\right]\cdot \E_{\half+\eps}\left[\text{down $c+c'$ tails or up $c-c'$ heads}\right] \\
        \Delta H^{+} + \Delta H^{-} 
        &= \Pr_{\half+\eps}\left[\text{coin arrives at }x\right]\cdot \frac{a^{-c'}\left(c(\alpha^c-1)(\alpha^{c'}+1)-c'(a^c+1)(a^{c'}-1)\right)}{2\eps(\alpha^c+1)}.
    \end{align*}
    \endgroup
    Consider first the case that $\Pr_{\half+\eps}\left[\text{coin arrives at }x\right]=0$, meaning there is no trajectory that can arrive at $x$. In this case, $\Delta \delta^{+} + \Delta \delta^{-} = \Delta H^+ + \Delta H^{-} = 0$, and so, the conclusion is true.

    Otherwise, if we take the ratio $\frac{\Delta \delta^{+} + \Delta \delta^{-}}{\Delta H^{+} + \Delta H^{-}}$, the common factor $\Pr_{\half+\eps}\left[\text{coin arrives at }x\right]$ cancels out, %
    yielding
    \begingroup
    \allowdisplaybreaks
    \begin{align}
        \frac{\Delta \delta^{+} + \Delta \delta^{-}}{\Delta H^{+} + \Delta H^{-}} 
        &= \frac{2\eps(1-\alpha^c)(\alpha^c-\alpha^{c'})}{(c+c')(\alpha^{2c}+\alpha^{c'})+(c-c')(\alpha^{2c+c'}+1)-2c\alpha^c(1+\alpha^{c'})} \nonumber \\
        &= \frac{2\eps(\alpha^c-\alpha^{c'})}{c'(\alpha^c+1)(\alpha^{c'}-1)-c(\alpha^c-1)(\alpha^{c'}+1)} := g(c'), \label{eqn:g-definition}
    \end{align}
    \endgroup
    where $g(c')$ is a function of the point $x$ we are attempting to erase. At this point, we require the following claim, whose proof is a tedious but elementary calculation, and is deferred to \Cref{sec:proof-derivative-calculation}.

    \begin{restatable}{claim}{claimderivative}
        \label{claim:derivative-calculation-erase}
        For the function $g(c')$ defined in \eqref{eqn:g-definition}, it is the case that $\frac{dg(c')}{dc'} > 0$.
    \end{restatable}
    
    \Cref{claim:derivative-calculation-erase} implies that $g(c')$ increases as $c'$ increases from $0$ to $c-1$. Furthermore, notice that $g(c')=\frac{\Delta \delta^{+} + \Delta \delta^{-}}{\Delta H^{+} + \Delta H^{-}} < 0$ for all $c'$, and hence it is largest when $c'=c-1$. 
    
    Now, for the action of erasing $x$ to be an improvement, we want
    \begin{align*}
        &\beta \cdot (\Delta H^{+} + \Delta H^{-}) \le -(\Delta \delta^{+} + \Delta \delta^{-}) \\
        \implies \qquad & \beta \le -\frac{\Delta \delta^{+} + \Delta \delta^{-}}{\Delta H^{+} + \Delta H^{-}} \qquad (\text{since $\Delta H^{+} + \Delta H^{-} \ge 0$}).
    \end{align*}
    By the preceding argument about the monotonicity of $g(\cdot)$, it is sufficient to choose
    \begin{align*}
    \beta \le -g(c-1) = \frac{2\eps\alpha^c(\alpha-1)}{(\alpha^c-\alpha)(\alpha^c+1)+2c\alpha^c(\alpha-1)},
    \end{align*}
    which is the upper bound claimed in \eqref{eqn:beta-ub}. By symmetry, the analysis is identical for case (2), concluding the proof.

\end{proof}

\begin{remark}[$\beta$ Ranges are Contiguous]
    \label{remark:beta-values-range}
    Let $l_c$ and $u_c$ denote the lower bound \eqref{eqn:beta-lb} and upper bound \eqref{eqn:beta-ub} on $\beta=\beta(c,\eps)$ respectively for a given value of $c, \eps$. Then, we can verify that $l_c < u_c$ for every $c \ge 1$, and that $u_1=\eps$. It is also the case that $l_c = u_{c+1}$ for every $c \ge 1$, meaning that the intervals $[l_c, u_c]$ are contiguous segments that span the interval $(0, \eps]$, where each pair $[l_c, u_c]$ and $[l_{c+1}, u_{c+1}]$ shares precisely one point, namely $u_{c+1}$.
\end{remark}

\subsection*{Putting everything together}
\label{sec:step-6-put-together}

We are now ready to complete the proof of our main \Cref{thm:linear-policy-beta}. The strategy is the following: after completing Steps 1 and 2 above, we are left with a truncated linear policy $\mcP_{c, n+1}$, whose Bayes risk at most $\gamma$ larger than the Bayes risk of the original policy $\mcP$. Thereafter, we run Steps 3, 4, 5 in loop, iteratively transforming $\mcP_{c, n+1}$ to $\mcP_{c, n}, \mcP_{c, n-1},\dots, \mcP_{c,0}$, while ensuring throughout that the Bayes risk does not increase (see \Cref{fig:step-6-put-together}). Since $\mcP_{c,0}$ is in fact the target linear policy $\mcP_c$, we will have shown that the Bayes risk of $\mcP_c$ is at most $\gamma$ more than that of $\mcP$. Since we can choose $\gamma$ to be arbitrarily small (which decides the truncation level $n$ to start with), we can conclude that the Bayes risk of $\mcP_c$ is at most the Bayes risk of $\mcP$.

\begin{figure}[H]
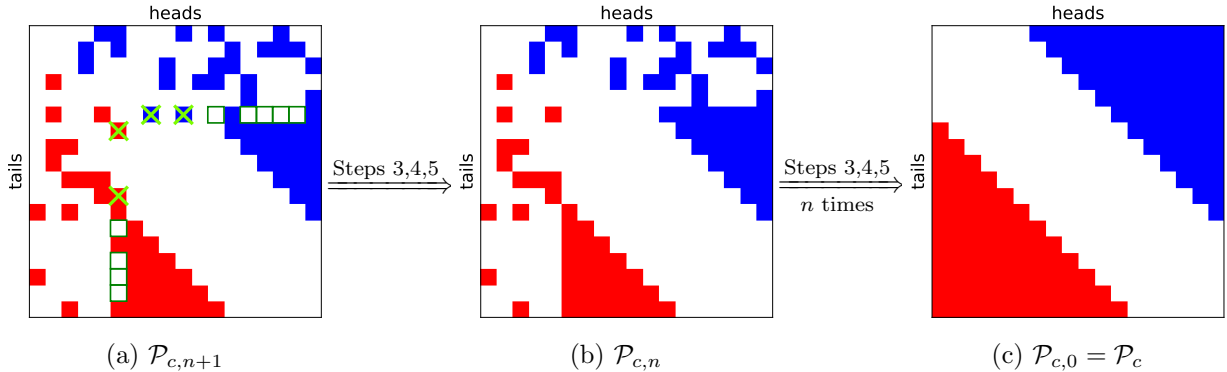

    \centering
    \begin{adjustbox}{center}
    \begin{tikzpicture}[baseline=(current bounding box.center)]
        \node[anchor=south west, inner sep=0] (img1) at (0,0) {
            \begin{minipage}{0.265\textwidth}
                \centering
                \includegraphics[width=\linewidth]{img/policy-linearized-all-moves-highlighted.pdf}
                \subcaption{$\mcP_{c, n+1}$}
            \end{minipage}
        };

        \node (arrow) at (5.2,2.5) {$\xRightarrow[\text{}]{\text{Steps 3,4,5}}$};

        \node[anchor=south west, inner sep=0] (img2) at (6,0) {
            \begin{minipage}{0.265\textwidth}
                \centering
                \includegraphics[width=\linewidth]{img/policy-linearized-erased-2-both-triangles.pdf}
                \subcaption{$\mcP_{c,n}$}
            \end{minipage}
        };

        \node (arrow) at (11.2,2.5) {$\xRightarrow[\text{$n$ times}]{\text{Steps 3,4,5}}$};

        \node[anchor=south west, inner sep=0] (img2) at (12,0) {
            \begin{minipage}{0.265\textwidth}
                \centering
                \includegraphics[width=\linewidth]{img/policy-fully-linearized.pdf}
                \subcaption{$\mcP_{c,0}=\mcP_{c}$}
            \end{minipage}
        };
    \end{tikzpicture}
    \end{adjustbox}
    \caption{Converting the truncated linear policy $\mcP_{c,n+1}$ step-by-step to the target linear policy $\mcP_c$ by repeating Steps 3,4,5.}
    \label{fig:step-6-put-together}
\end{figure}

We restate \Cref{thm:linear-policy-beta} here for convenience.

\theoremlinearpolicyoptimalitybeta*
\begin{proof}
    Let $\set{P}$ be any policy with profile $(\delta, H)$. We consider two cases:
    \paragraph{Case 1:} $\beta \le \eps$.\\
    \noindent We first perform Step 1 on $\mcP$. Then, we consider the linear policy $\set{P}_c$ corresponding to a $c  \in \mathbb{Z}_{\ge 1}$ satisfying
    \begin{align}
        \frac{2\eps\alpha^c(\alpha-1)}{(\alpha^{c+1}-1)(\alpha^c+1)+2c\alpha^c(\alpha-1)} \le \beta \le \frac{2\eps\alpha^c(\alpha-1)}{(\alpha^c-\alpha)(\alpha^c+1)+2c\alpha^c(\alpha-1)}. \label{eqn:beta-sandwich}
    \end{align}
    Here, the upper bound is from \eqref{eqn:beta-ub} and the lower bound is from \eqref{eqn:beta-lb}. Since $\beta \le \eps$, we know from \Cref{remark:beta-values-range} that the choice of $c$ above is well-defined. 
    For any $\gamma > 0$, \Cref{claim:step-2-truncation-to-linear} then guarantees the existence of a \textit{truncated} linear policy $\set{P}_{c,n+1}$ for a large enough (but finite) $n$, having profile $(\delta_{c, n+1}, H_{c, n+1})$, such that
    \begin{align*}
        (\delta^+_{c,n+1} + \delta^-_{c,n+1}) + \beta(H^+_{c,n+1} + H^-_{c,n+1}) \le (\delta^+ + \delta^-) + \beta(H^+ + H^-) + \gamma.
    \end{align*}
    We next perform Step 2, by truncating $\mcP$ to such a $\mcP_{c, n+1}$. Our strategy now is to run Steps 3, 4, 5 in loop starting with $\set{P}_{c,n+1}$, so that we end up with the linear policy $\set{P}_{c}$. In more detail, we first perform Steps 3, 4, 5 in both the upper and lower triangles of $\set{P}_{c, n+1}$. Note that at this point, we will have converted the policy $\set{P}_{c,n+1}$ that we started with to the policy $\set{P}_{c,n}$. We can then repeat this process till we end up with $\set{P}_{c,0}=\set{P}_c$. Throughout the process, we will have ensured that the Bayes risk does not increase. Namely, chaining invocations of \Cref{claim:step-1-fix-colors}, \Cref{claim:step-2-truncation-to-linear}, \Cref{claim:step-3-fill-gaps}, \Cref{lemma:step-4-place-hitting-point} and \Cref{lemma:step-5-erase-extra-points}, we have that
    \begin{align*}
        (\delta^+_{c} + \delta^-_{c}) + \beta(H^+_{c} + H^-_{c}) \le (\delta^+ + \delta^-) + \beta(H^+ + H^-) + \gamma.
    \end{align*}
    Finally, since $\gamma > 0$ is arbitrary, it must be that
    \begin{align*}
        (\delta^+_{c} + \delta^-_{c}) + \beta(H^+_{c} + H^-_{c}) \le (\delta^+ + \delta^-) + \beta(H^+ + H^-).
    \end{align*}

    \paragraph{Case 2:} $\beta > \eps$. \\
    \noindent Note that $\beta=\eps$ is a valid value of $\beta$ for $c=1$. So, by the analysis in Case 1 above, we have that for the linear policy $\set{P}_1$ corresponding to $c=1$, 
    \begin{align*}
        (\delta^+_{1} + \delta^-_{1}) + \eps(H^+_{1} + H^-_{1}) \le (\delta^+ + \delta^-) + \eps(H^+ + H^-).
    \end{align*}
    But observe that $\set{P}_1$ is simply the policy that tosses once, and declares $\half+\eps$ if the coin lands heads, and declares $\half-\eps$ if it lands tails. In particular, $\delta^+_1 = \delta^-_1=\half-\eps$, and $H^+_1=H^-_1=1$, and so
    \begin{equation}
        \label{eqn:compare-to-c1}
        1-2\eps+2\eps=1 \le (\delta^+ + \delta^-) + \eps(H^+ + H^-).
    \end{equation}
    Now consider the policy which declares $\half+\eps$ always, without even tossing the coin at all. This would technically correspond to the linear policy $\set{P}_0$ which has stopping points at $h-t=0$---in particular, it stops at $(h,t)=(0,0)$ itself. For this policy, we have that $\delta^+_0=0, \delta^-_0=1$, $H^+_0=H^-_0=0$, and so, for any $\beta$,
    \begin{align*}
        (\delta^+_0 + \delta^-_0) + \beta(H^+_0 + H^-_0) &= 1 \\
        &\le (\delta^+ + \delta^-) + \eps(H^+ + H^-) \qquad (\text{from \eqref{eqn:compare-to-c1}})\\
        &\le (\delta^+ + \delta^-) + \beta(H^+ + H^-). \qquad (\text{since $\beta>\eps$})
    \end{align*}
    Thus, we have shown that for any $\beta > 0$, we can find the linear policy $\mcP_c$ corresponding to the segment $[l_c, u_c]$ in which $\beta$ lies, where $l_c, u_c$ are the lower and upper bounds in \eqref{eqn:beta-sandwich} respectively, and this linear policy is optimal with respect to $\beta$.
\end{proof}

\begin{remark}
    \label{remark:optimal-policy-characterization}
    The characterization of the precise policy $\mcP_c$ that satisfies the optimality criterion \eqref{eqn:bayes-risk-optimality-criterion} in \Cref{thm:linear-policy-beta} for a given value of $\beta$ is as follows: if $\beta \le \eps$, then $\mcP_c$ is the linear policy, whose intercept $c \ge 1$ satisfies \eqref{eqn:beta-sandwich}.
    On the other hand, if $\beta > \eps$, then $\mcP_0$, i.e., the policy that immediately declares $\half+\eps$ (or $\half-\eps$) right at the outset, without tossing at all, satisfies the optimality criterion. %
\end{remark}

\section{Generalizations and Future Work}
\label{sec:generalizations}

In this work, we presented a geometric, local moves framework to show the optimality of the SPRT. Specifically, for the task of distinguishing between symmetric Bernoulli hypotheses, our proof characterized the linear policy $\mcP_c$ that optimizes the Bayes risk $R(\beta)$ for every parameter $\beta$. %

As noted after the statement of \Cref{thm:linear-policy-beta}, our analysis considered a special class of the general Bayes risk $R(w_0, w_1, C, \pi)$ given in \eqref{eqn:bayes-risk}, since it assumes $w_0=w_1$ and a uniform prior $\pi$. It is known that a linear policy is optimal even in the more general case \cite{girshick1946contributions}. It would be interesting to see if our algorithmic proof can be extended to this more general setting. We can first consider a uniform prior, but different weights $w_0$ and $w_1$, before considering arbitrary priors $\pi$. In these cases, however, we lose all the convenient symmetries in the Bayes risk that we relied upon at various places in our proof.  

A different direction would be to consider the problem of distinguishing between arbitrary, non-symmetric Bernoulli hypotheses $H_0:p=p_1$ versus $H_1:p=p_2$ where $p_1 < p_2$. While the framework of local improvements that we rely on may still work, the bounds on the parameters would become significantly more complicated, requiring messier algebraic derivations. In particular, we can no longer cleanly apply transformations relating $\Pr_{H_0}[\text{coin arrives at } x]$ to $\Pr_{H_1}[\text{coin arrives at } x]$ such as the one we used in \eqref{eqn:calc-relating-plus-minus}, when $H_0$ and $H_1$ were symmetric Bernoulli hypotheses. Lastly, it would also be interesting to see if this local moves framework gives something for composite hypothesis testing tasks, like testing $H_0:p \le 1/2-\eps$ versus $H_1:p\ge 1/2+\eps$.

\section*{Acknowledgements}
This work was supported by Gregory Valiant's and Moses Charikar's Simons Investigator Awards.

\bibliographystyle{alpha}
\bibliography{references}

\appendix
\section{Proof of \Cref{claim:derivative-calculation-erase}}
\label{sec:proof-derivative-calculation}

We restate \Cref{claim:derivative-calculation-erase} for convenience:
\claimderivative*
\begin{proof}
    Recall the definition of the function $g(c')$:
    \begin{align*}
        g(c') := \frac{2\eps(\alpha^c-\alpha^{c'})}{c'(\alpha^c+1)(\alpha^{c'}-1)-c(\alpha^c-1)(\alpha^{c'}+1)}.
    \end{align*}
    Here, $c > c'$.
    We can verify that
    \begin{align}
        \frac{dg(c')}{dc'} &= 2\eps(1-\alpha^c) \cdot -\frac{(\alpha^c+1)[\ln(\alpha)\cdot(\alpha^c-1)(c-c')\alpha^{c'}-(\alpha^{c'}-1)(\alpha^c-\alpha^{c'})]}{(\alpha^c-1)[c'(\alpha^c+1)(\alpha^{c'}-1)-c(\alpha^c-1)(\alpha^{c'}+1)]^2} \nonumber \\ 
        &= 2\eps \cdot \frac{(\alpha^c+1)[\ln(\alpha)\cdot(\alpha^c-1)(c-c')\alpha^{c'}-(\alpha^{c'}-1)(\alpha^c-\alpha^{c'})]}{[c'(\alpha^c+1)(\alpha^{c'}-1)-c(\alpha^c-1)(\alpha^{c'}+1)]^2}. %
        \label{eqn:outward-move-derivative-positive}
    \end{align}
    We will first argue that the denominator is positive. Since it is a square, it suffices to argue that the denominator is not $0$. Note first that if $c'=0$, then the denominator is positive. Now, assume that $c'>0$ but the denominator is still 0. Then, it must be so that
    $$\frac{\alpha^c+1}{\alpha^c-1} = \frac{c}{c'}\cdot\frac{\alpha^{c'}+1}{\alpha^{c'}-1},$$
    but observe that since $c > c'$, both $\frac{c}{c'} > 1$ and $\frac{\alpha^{c'}+1}{\alpha^{c'}-1} > \frac{\alpha^c+1}{\alpha^c-1}$; hence, this cannot be the case. 

    Now, we will reason about the numerator. In particular, we will show that for any $0 \le c' < c$, the numerator is strictly positive. Note that it suffices to show that the function $$\ln(\alpha)\cdot(\alpha^c-1)(c-c')\alpha^{c'}-(\alpha^{c'}-1)(\alpha^c-\alpha^{c'})$$ is positive. 

    The derivative of this function with respect to $c$ is 
    $$\ln(\alpha) \left(\alpha^c - \alpha^{c'} + \alpha^{c + c'} (c - c') \ln(\alpha)\right).$$
    Note that the derivative is positive, so this function is increasing in $c$. Thus, it suffices to show that the function is positive for $c = c' + 1$. Substituting, we get
    $$(\ln \alpha)(\alpha^c-1)\alpha^{c-1} - (\alpha^c-\alpha^{c-1})(\alpha^{c-1}-1).$$
    We only care about the sign of this expression, so we can divide by $a^{c-1}$, to obtain
    $$(\ln \alpha)(\alpha^c-1) - (\alpha-1)(\alpha^{c-1}-1).$$
    This function is smooth. In particular, this function can only change its sign at a root. We can observe that $\alpha=1$ is a root. We claim there are no other roots. 
    
    \noindent When $\alpha \neq 1$, we can solve for $c$ to get
    $$c = \frac{\ln\left(\frac{\alpha - \alpha^2 + \alpha\ln(\alpha)}{1 - \alpha + \alpha\ln(\alpha)}\right)}{\ln(\alpha)}.$$
    
    \noindent Note that a solution will exist for $c$ if and only if the expression inside the log is positive. In particular, we claim that 
    $$\frac{\alpha - \alpha^2 + \alpha\ln(\alpha)}{1 - \alpha + \alpha\ln(\alpha)} = \alpha\cdot\frac{1 - \alpha + \ln(\alpha)}{1 - \alpha + \alpha\ln(\alpha)}$$
    is negative.
    
    \noindent The numerator is negative and the denominator is positive, so this expression is always negative. For the numerator, note that $e^x > x + 1$ for all positive $x$, and substitute $x = \ln \alpha$. For the denominator, note than its derivative is $\ln \alpha$, so the denominator is increasing when $\alpha>1$, so it is greater than when evaluated at 1, which is $0$.
    
    \noindent Thus, there are no roots other than at $\alpha=1$. Consequently, when $\alpha>1$, the sign of the function is the same as when it is evaluated at $\alpha=e$, where it is positive. 

    Summarily, we have argued that both the numerator and denominator in \eqref{eqn:outward-move-derivative-positive} are positive. This completes the proof that $\frac{dg(c')}{dc'} > 0$.

\end{proof}
\end{document}